\documentstyle[11pt,amssymb,amsmath,amscd,amsbsy,amsfonts,amsthm,color]{article}

\input xy
\xyoption{all}

\newtheorem{thm}{Theorem}[section]
\newtheorem{lem}[thm]{Lemma}
\newtheorem{prop}[thm]{Proposition}
\newtheorem{cor}[thm]{Corollary}
\newtheorem{rem}[thm]{Remark}
\newtheorem{dfn}[thm]{Definition}

\DeclareMathOperator{\Z}{\mathbb{Z}}
\DeclareMathOperator{\Gm}{{\mathbb G}_m}

\DeclareMathOperator{\DM}{\mathcal{DM}}

\DeclareMathOperator{\X}{\mathfrak{X}}
\DeclareMathOperator{\A}{\mathcal{A}}

\DeclareMathOperator{\Pa}{\mathrm{P}}

\DeclareMathOperator{\bock}{\mathrm{B}}

\DeclareMathOperator{\Hom}{\mathrm{Hom}}

\DeclareMathOperator{\Ima}{im}
\DeclareMathOperator{\coker}{coker}
\DeclareMathOperator{\seb}{SB}

\title{\textsc{A Serre-type spectral sequence for motivic cohomology}}
\author{Fabio Tanania}
\date{}

\begin{document}

\maketitle

\begin{abstract}
In this paper, we construct and study a Serre-type spectral sequence for motivic cohomology associated to a map of bisimplicial schemes with motivically cellular fiber. Then, we show how to apply it in order to approach the computation of the motivic cohomology of the Nisnevich classifying space of projective general linear groups. This naturally yields an explicit description of the motive of a Severi-Brauer variety in terms of twisted motives of its \v{C}ech simplicial scheme.
\end{abstract}

\section{Introduction}
\label{sec:introduction}

One of the greatest features of topology is the plenty of powerful computational tools. The Serre spectral sequence associated to a fibration is a notable example that allows to extract information about the cohomology of the total space, once known the cohomology of the base and of the fiber. As a particular case, when the fiber is a sphere, the Serre spectral sequence gives back the Gysin long exact sequence associated to a sphere bundle.

The development of motivic homotopy theory made it possible to efficiently import topological techniques into the algebro-geometric world. This translation is usually not straightforward nor even univoque (e.g. unlike the topological picture, in motivic homotopy theory we have several classifying spaces, each one corresponding to a different Grothendieck topology). Most of the times, in fact, it is not clear how the topologically inspired tool should look like in the motivic world and what assumptions are needed in order to make the translation successful.

In this paper, we would like to propose the construction of a Serre-type spectral sequence for motivic cohomology. This spectral sequence arises from a Postnikov system in a certain triangulated category of motives associated to a morphism of bisimplicial schemes whose fiber is motivically cellular, i.e. a direct sum of Tate motives. As in topology, our Serre spectral sequence reconstructs the motivic cohomology of the total space out of the motivic cohomology of the base and the cellular structure of the fiber. It is a natural generalisation of the Gysin long exact sequence that was first introduced by Smirnov and Vishik in \cite{smirnov.vishik} for computing the motivic cohomology of the Nisnevich classifying space of orthogonal groups, and subsequently studied by the author in \cite{tanania.b} for the general case of morphisms of simplicial schemes with reduced Tate fiber (the motivic analogues of sphere bundles). For constructing our spectral sequence, we need to work with bisimplicial schemes instead of simplicial ones. Indeed, we want to allow simplicial fibers that are in some sense infinite-dimensional. We also want to point out that a motivic Serre spectral sequence of different nature has been recently developed by Asok, D\'eglise and Nagel in \cite{asok}. A significant difference resides in the fact that, while the latter is achieved by filtering with respect to the $\delta$-homotopy $t$-structure whose slices are homotopy modules, in our spectral sequence the slices are, instead, Tate motives. 

The Gysin long exact sequence for motivic cohomology was successfully exploited for computing the motivic cohomology ring of the Nisnevich classifying space $BG$ for some linear algebraic groups $G$, which in turn produced new subtle invariants for $G$-torsors. These new invariants, unlike the usual ones coming from the motivic cohomology of the \'etale classifying space $B_{\acute et}G$, take values in a more informative object, namely the motivic cohomology of the \v{C}ech simplicial scheme of the torsor under study.

More precisely, the Gysin sequence was used for orthogonal groups in \cite{smirnov.vishik}, for unitary groups in \cite{tanania.a} and for spin groups in \cite{tanania.b}. Unfortunately, some algebraic groups are out of its range of application, e.g. the projective general linear group $PGL_n$. The cohomology of its classifying space is notoriously difficult to study even in topology where, neverthless, some important results are known thanks to the use of the Serre spectral sequence. For example, the singular cohomology of $BPU_n$ has been computed in low degrees by Antieau and Williams in \cite{antieau.williams} and by Gu in \cite{gu1}. Moreover, the singular cohomology groups and the Chow groups of the \'etale classifying space of $PGL_p$ for an odd prime $p$ were completely computed by Vistoli in \cite{vistoli}. The Nisnevich classifying space of $PGL_n$ was first studied by Rolle in \cite{rolle}. In this paper, we show how to apply our spectral sequence to obtain some information about the motivic cohomology of $BPGL_n$. Indeed, we use the Serre spectral sequence developed here for the obvious map $BGL_n \rightarrow BPGL_n$ with fiber $B\Gm$ in order to compute the motivic cohomology of $BPGL_n$ in low motivic weights.

As noted before, investigating classifying spaces from a cohomological point of view is essential for the program of classifying torsors of linear algebraic groups. For example, the importance of studying Nisnevich classifying spaces for the classification of quadratic forms was highlighted by Smirnov and Vishik in \cite{smirnov.vishik}. In the case of projective general linear groups, torsors are central simple algebras and a better understanding of the Nisnevich classifying space of $PGL_n$ provides information about them. 

Motivic descriptions of geometric objects related to central simple algebras have been widely investigated in the last decades. For example, Karpenko shows in \cite{karpenko} that the Chow motive of the Severi-Brauer variety associated to a central simple algebra is the direct sum of twisted copies of the motive of the Severi-Brauer variety associated to the underlying division algebra, while the latter is indecomposable. In \cite{kahn.levine}, Kahn and Levine obtain a Postnikov tower in the Voevodsky's triangulated category of motives for the Severi-Brauer variety associated to a division algebra of prime degree. In \cite{shinder}, Shinder studies the slice filtration for the motive of the group of units of a division algebra of prime degree, which involves the motive of the associated \v{C}ech simplicial scheme. In this paper, we show how to apply the techniques developed here to obtain a description of the motive of a Severi-Brauer variety associated to any central simple algebra in terms of the respective \v{C}ech simplicial scheme. The latter description also induces a spectral sequence, whose first differential is fully understood, strongly converging to the motivic cohomology of the Severi-Brauer variety starting from the motivic cohomology of its \v{C}ech simplicial scheme.

As we have already pointed out, a complete description of the singular cohomology of $BPU_n$ seems to be still out of reach. Neverthless, in \cite{gu2} and \cite{gu3} Gu has constructed non-trivial torsion classes in the cohomology ring of $BPU_n$ and in the Chow ring of $B_{\acute et}PGL_n$ over the complex numbers, by using Steenrod operations and the cycle class map going from Chow groups to singular cohomology. These important results shed new light on the structure of these rings. In this paper, by following the lead of \cite{gu2} and \cite{gu3}, we show that there are similar non-trivial torsion classes in the motivic cohomology of the Nisnevich classifying space of $PGL_n$. By exploiting the homomorphism in motivic cohomology induced by the canonical map $BPGL_n \rightarrow B_{\acute et}PGL_n$ (that replaces the cycle class map over general fields), we then generalise Gu's results on the torsion classes in the Chow ring of $B_{\acute et}PGL_n$ to fields of characteristic not dividing $n$ containing a primitive $n$th root of unity.\\

\textbf{Outline.} In Section 2, we report the main notations we need in this paper. In Section 3, we recall a few things about Postnikov systems in triangulated categories and induced spectral sequences. Section 4 is devoted to the introduction of the triangulated category of motives over a bisimplicial scheme that is modeled on the one constructed for simplicial schemes by Voevodsky in \cite{voevodsky.simplicial}. Sections 5 and 6 contain the main results of this paper, i.e. the construction of the Serre spectral sequence for motivic cohomology (see Theorem \ref{Serre2}) and the investigation of its multiplicative structure (see Theorem \ref{mult}). In Section 7, we apply the latter to compute the motivic cohomology of the Nisnevich classifying space of $PGL_n$ in low motivic weights (see Theorem \ref{bpg}). Section 8 provides a Postnikov system for the motive of a Severi-Brauer variety (see Proposition \ref{pssb}). Finally, in Section 9, we find torsion classes in the motivic cohomology of $BPGL_n$ generalising some results from \cite{gu2} and \cite{gu3} (see Corollary \ref{gentor}).\\

\textbf{Acknowledgements.} I would like to thank Alexander Vishik and Olivier Haution for helpful conversations on the topic of this paper. I am also grateful to the referee for very useful comments that helped improving the overall exposition, and fixing some mistakes.
 
 \section{Notation}
 \label{sec:notation}
 
 Here we fix some notations we use throughout this paper.
 
 \begin{tabular}{c|c}
 	$k$ & infinite perfect field\\
 	$R$ & commutative ring with identity\\
 	$Y_{\bullet,\bullet}$ & smooth bisimplicial scheme over $k$\\
 	$d(Y_{\bullet,\bullet})$ & diagonal simplicial scheme of $Y_{\bullet,\bullet}$\\
 	$\DM^-_{eff}(k,R)$ & triangulated category of motives over $k$ with $R$-coefficients \\
 	$\DM^-_{eff}(Y_{\bullet,\bullet},R)$ & triangulated category of motives over $Y_{\bullet,\bullet}$ with $R$-coefficients\\
 	$\DM_{coh}^-(Y_{\bullet,\bullet},R)$ & localizing subcategory of coherent motives over $Y_{\bullet,\bullet}$ with $R$-coefficients\\
 	$T$ & unit object in $\DM^-_{eff}(k,R)$\\
 	$H^{**}(-,R)$ & motivic cohomology with $R$-coefficients\\
 	$H^{**}(-)$ & motivic cohomology with $\Z$-coefficients\\
 	$CH^*(-)$ & Chow groups, i.e. $CH^q(-) \cong H^{2q,q}(-)$\\
 	${\mathcal H}_s(k)$ & simplicial homotopy category over $k$\\
 	${\mathcal H}(k)$ & $A^1$-homotopy category over $k$\\
 	$BG_{\bullet}$ & Nisnevich classifying space of the simplicial algebraic group $G_{\bullet}$\\
 	$B_{\acute et}G$ & \'etale classifying space of the algebraic group $G$\\
 	$PGL_n$ & projective general linear group\\
 	$A$ & central simple algebra over $k$\\
 	$\seb(A)$ & Severi-Brauer variety associated to $A$\\
 	$\X_A$ & motive of the \v{C}ech simplicial scheme of $\seb(A)$\\
 \end{tabular}\\
 
 \section{Some general facts about spectral sequences}\label{spec}
 
 Let us start by recalling some well known facts about Postnikov systems in triangulated categories, spectral sequences associated to them and convergence issues (see \cite{boardman}, \cite{gelfand.manin} and \cite{mccleary}).
 
 Throughout this section, we denote by $\mathcal{C}$ a triangulated category, by $\A$ an abelian category and by $H:\mathcal{C} \rightarrow \A$ a cohomological functor, i.e. an additive contravariant functor sending distinguished triangles in $\mathcal{C}$ into long exact sequences in $\A$.
 
 \begin{dfn}\label{ec}
 	An exact couple in $\A$ is a triangle
 	$$
 	\xymatrix{
 		D  \ar@{->}[rr]^{i} & & D \ar@{->}[dl]^{j}\\
 		& E \ar@{->}[ul]^{k}&
 	}
 	$$
 	which is exact at each vertex.
 \end{dfn}
 
 The morphism $d$ defined as the composition $jk$ is a differential, i.e. $d^2=0$. Set $E'=\ker(d)/\Ima(d)$ and $D'=\Ima(i)$. One can define morphisms $j':D' \rightarrow E'$ and $k':E' \rightarrow D'$ respectively by $j'(i(x))=j(x)$ and $k'([y])=k(y)$ for any $x \in D$ and $y \in \ker(d)$. The triangle obtained in this way
 $$
 \xymatrix{
 	D'  \ar@{->}[rr]^{i} & & D' \ar@{->}[dl]^{j'}\\
 	& E' \ar@{->}[ul]^{k'}&
 }
 $$
 is again an exact couple, called the derived couple. Reiterating this construction, one gets a sequence of objects $E_r$ in $\A$, endowed with differentials $d_r$, each of which is the homology of the previous one. More precisely, we can give the following definition.
 
 \begin{dfn}
 	The sequence $\{E_r,d_r\}_{r \geq 1}$ constructed inductively by
 	$$E_r=\ker(d_{r-1})/\Ima(d_{r-1})$$
 	is called the spectral sequence associated to the exact couple in Definition \ref{ec}.
 \end{dfn}
 
 In practical situations, one often encounters bigraded exact couples, which naturally give rise to bigraded spectral sequences $\{E_r^{s,t},d_r^{s,t}\}_{r \geq 1}$. We now provide the definition of the limit page of a bigraded spectral sequence.
 
 \begin{dfn}
 	Let $\{E_r^{s,t},d_r^{s,t}\}_{r \geq 1}$ be a bigraded spectral sequence and suppose there is an integer $r(s,t)$ such that $E_r^{s,t} \cong E_{r(s,t)}^{s,t}$ for any $r \geq r(s,t)$, then we say that the spectral sequence abuts to $E_{\infty}^{s,t}=E_{r(s,t)}^{s,t}$.
 \end{dfn}
 
 At this point, let us recall some notions about filtrations and covergence of spectral sequences from \cite{boardman}.
 
 An increasing filtration of an object $G$ in $\A$ is a diagram of the following shape 
 $$\dots \hookrightarrow F^1 \hookrightarrow F^2 \hookrightarrow \dots \hookrightarrow F^m \hookrightarrow F^{m+1} \hookrightarrow \dots \hookrightarrow G.$$
 
 \begin{dfn}
 	The increasing filtration $\{F^m\}_{m \in \Z}$ of  $G$ is said to be:\\
 	1) exhaustive if $G=\varinjlim_m F^m$;\\
 	2) Hausdorff if $\varprojlim_m F^m=0$;\\
 	3) complete if $\varprojlim^1_m F^m=0$.
 \end{dfn}
 
 In practice, one is often interested in filtrations of graded objects. So, for any $F^s$ in the filtration we denote by $F^{s,t}$ its graded component in degree $t$.
 
 \begin{dfn}
 	A spectral sequence associated to an exact couple is called:\\
 	1) weakly convergent to $G$ if there exists an increasing filtration of $G$ which is exhaustive and such that $E^{s,t}_{\infty} \cong F^{s,t}/F^{s-1,t}$ for any $s$;\\
 	2) convergent to $G$ if it is weakly convergent and the filtration of $G$ is Hausdorff;\\
 	3) strongly convergent to $G$ if it is weakly convergent and the filtration of $G$ is complete Hausdorff.
 \end{dfn}
 
 We now recall the definition of Postnikov system in a triangulated category (see \cite{gelfand.manin}). 
 
 \begin{dfn}\label{ps}
 	A Postnikov system for an object $X$ in $\mathcal{C}$ is a diagram
 	$$
 	\xymatrix{
 		\dots \ar@{->}[r] & X_{i+1} \ar@{->}[r] \ar@{->}[d]	 &X_i \ar@{->}[r] \ar@{->}[d] &  \dots \ar@{->}[r]   & X_2 \ar@{->}[r] \ar@{->}[d]	 & X=X_1  \ar@{->}[d]  \\
 		& Y_{i+1} \ar@{->}[ul]^{[1]} &	Y_i \ar@{->}[ul]^{[1]} & & Y_2 \ar@{->}[ul]^{[1]}  &	Y_1 \ar@{->}[ul]^{[1]} 
 	}
 	$$
 	where all the triangles are distinguished triangles in $\mathcal{C}$.
 \end{dfn}
 
 Associated to a Postnikov system one can always construct an exact couple by applying a cohomological functor $H$. More precisely, we have the following bigraded exact couple
 $$
 \xymatrix{
 	D  \ar@{->}[rr]^{i} & & D \ar@{->}[dl]^{j}\\
 	& E \ar@{->}[ul]^{k}&
 }
 $$
 where $D^{s,t}=H^t(X_s): =H(X_s[-t])$, $E^{s,t}=H^t(Y_s): =H(Y_s[-t])$ and the morphisms $i:D^{s,t} \rightarrow D^{s+1,t}$, $j:D^{s,t} \rightarrow E^{s-1,t+1}$ and $k:E^{s,t} \rightarrow D^{s,t}$ are induced by the morphisms in the Postnikov system.
 
 As usual, we obtain an increasing filtration of the object $H^t(X)$ in $\A$ given by 
 $$F^{1,t} \hookrightarrow F^{2,t} \hookrightarrow \dots \hookrightarrow F^{s-1,t} \hookrightarrow F^{s,t} \hookrightarrow \dots \hookrightarrow H^t(X)$$
 where $F^{s,t}=\ker(H^t(X) \rightarrow H^t(X_s))$ and the morphism $H^t(X) \rightarrow H^t(X_s)$ is the one induced by the Postnikov system. Moreover, observe that the filtration $\{F^{s,t}\}_{s \geq 1}$ just introduced is complete Hausdorff, since it is bounded from below, but not necessarily exhaustive.	Anyway, we have the following result that guarantees the strong convergence of the spectral sequence just constructed, provided that a certain condition holds.
 
 \begin{thm}\label{SC}
 	If $\varinjlim_s H^*(X_s) \cong 0$, then the spectral sequence associated to the Postnikov system in Definition \ref{ps} is strongly convergent to $H^*(X)$.
 \end{thm}
 \begin{proof}
 	See \cite[Theorem 6.1]{boardman}.
 \end{proof}
 
 \section{Motives over a bisimplicial scheme}
 
 For technical reasons, in this paper we need to work over bisimplicial schemes. To this end, we need a triangulated category of motives over a bisimplicial scheme. The triangulated category of motives over a simplicial scheme was introduced and studied by Voevodsky in \cite{voevodsky.simplicial}. More recently, Cisinski and D\'eglise constructed in \cite{cisinski.deglise} a triangulated category of motives over an arbitrary diagram of schemes indexed by a small category. We point out that the constructions and results we need from \cite{voevodsky.simplicial} extend to the bisimplicial case in a straightforward way. In this section we briefly summarise them. 
 
 Let $Y_{\bullet,\bullet}$ be a smooth bisimplicial scheme over $k$. Following \cite[Section 2]{voevodsky.simplicial}, we define $Sm/Y_{\bullet,\bullet}$ in the following way.
 
 \begin{dfn}
 	Denote by $Sm/Y_{\bullet,\bullet}$ the category whose objects are triples $(U,i,h)$, where $i$ and $h$ are non-negative integers and $U$ is a smooth scheme over $Y_{i,h}$, and whose morphisms from $(U,i,h)$ to $(V,j,k)$ are triples $(f,\phi,\psi)$, where $\phi:[j] \rightarrow [i]$ and $\psi:[k] \rightarrow [h]$ are simplicial maps and $f:U \rightarrow V$ is a morphism of schemes such that the square
 	$$
 	\xymatrix{
 		U \ar@{->}[r]^{f} \ar@{->}[d] & V \ar@{->}[d]\\
 		Y_{i,h} \ar@{->}[r]_{Y_{\phi,\psi}} & Y_{j,k}
 	}
 	$$
 	commutes.
 \end{dfn}
 
 We can also define presheaves on $Y_{\bullet,\bullet}$ following \cite[Definition 2.1]{voevodsky.simplicial}.
 
 \begin{dfn}
 	A presheaf of sets (respectively with transfers) on $Y_{\bullet,\bullet}$ consists of a collection $\{F_{i,h}\}_{i,h \geq 0}$ of presheaves of sets (respectively with transfers) on $Sm/Y_{i,h}$ together
 	with a morphism of presheaves of sets (respectively with transfers) $F_{\phi,\psi}: Y_{\phi,\psi}^*(F_{j,k}) \rightarrow  F_{i,h}$ for any simplicial maps $\phi:[j] \rightarrow [i]$ and $\psi:[k] \rightarrow [h]$, such that $F_{id,id} =id$ and $F_{\phi\alpha,\psi\beta} :Y_{\phi\alpha,\psi\beta}^*(F_{m,n}) \rightarrow F_{i,h}$ is equal to the composition of $Y_{\phi,\psi}^*F_{\alpha,\beta} : Y_{\phi,\psi}^*Y_{\alpha,\beta}^*(F_{m,n})  \rightarrow  Y_{\phi,\psi}^*(F_{j,k})$ and $F_{\phi,\psi} : Y_{\phi,\psi}^*(F_{j,k})  \rightarrow  F_{i,h}$, where $\alpha : [m]  \rightarrow  [j]$ and $\beta : [n]  \rightarrow  [k]$ are simplicial maps.
 \end{dfn}
 
 Denote by $PShv(Y_{\bullet,\bullet})$ the category of presheaves of sets on $Y_{\bullet,\bullet}$ and by $PST(Y_{\bullet,\bullet},R)$ the abelian category of presheaves of $R$-modules with transfers on $Y_{\bullet,\bullet}$.
 Note that $PShv(Y_{\bullet,\bullet})$ is nothing but the category of contravariant functors from $Sm/Y_{\bullet,\bullet}$ to $Sets$.
 
 If $F = \{F_{i,h}\}_{i,h \geq 0}$ is a presheaf of sets on $Y_{\bullet,\bullet}$, then $R_{tr}F = \{R_{tr}F_{i,h}\}_{i,h \geq 0}$ is a presheaf with transfers on $Y_{\bullet,\bullet}$. In particular, denote by $R_{tr}(U,i,h)$ the presheaf with transfers associated to the representable presheaf of sets corresponding to $(U,i,h)$.
 
 Let $SmCor(Y_{\bullet,\bullet},R)$ be the full subcategory of $PST(Y_{\bullet,\bullet},R)$ whose objects are arbitrary direct sums of objects of the form $R_{tr}(U,i,h)$.
 
 \begin{lem}
 	The category $PST(Y_{\bullet,\bullet},R)$ is naturally equivalent to the category of $R$-linear contravariant functors from $SmCor(Y_{\bullet,\bullet},R)$ to the category of $R$-modules that preserve coproducts.
 \end{lem}
 \begin{proof}
 	See \cite[Lemma 2.3]{voevodsky.simplicial}.
 \end{proof}
 
 The previous result allows, as usual, to construct left resolutions $Lres(F)$ consisting of representable presheaves with transfers for any $F$ in $PST(Y_{\bullet,\bullet},R)$.
 
 For any non-negative integers $i$ and $h$ denote by $r_{i,h} : SmCor(Y_{i,h},R)  \rightarrow SmCor(Y_{\bullet,\bullet},R)$ the functor that sends $U$ to $R_{tr}(U,i,h)$. These functors induce in the standard way pairs of adjoint functors
 \begin{align*}
 	PST&(Y_{\bullet,\bullet},R)\\
 	r_{i,h,\#} \uparrow & \downarrow r_{i,h}^*\\
 	PST&(Y_{i,h},R).
 \end{align*}
 
 There are similar functors $r_{i,\bullet} : SmCor(Y_{i,\bullet},R)  \rightarrow SmCor(Y_{\bullet,\bullet},R)$ sending $R_{tr}(U,h)$ to $R_{tr}(U,i,h)$, which induce pairs of adjoint functors
 \begin{align*}
 	PST&(Y_{\bullet,\bullet},R)\\
 	r_{i,\bullet,\#} \uparrow & \downarrow r_{i,\bullet}^*\\
 	PST&(Y_{i,\bullet},R),
 \end{align*}
 
 and $r_{\bullet,h} : SmCor(Y_{\bullet,h},R)  \rightarrow SmCor(Y_{\bullet,\bullet},R)$ sending $R_{tr}(U,i)$ to $R_{tr}(U,i,h)$, which induce pairs of adjoint functors
 \begin{align*}
 	PST&(Y_{\bullet,\bullet},R)\\
 	r_{\bullet,h,\#} \uparrow & \downarrow r_{\bullet,h}^*\\
 	PST&(Y_{\bullet,h},R).
 \end{align*}
 
 Finally, we can also consider the diagonal functor $d : SmCor(d(Y_{\bullet,\bullet}),R)  \rightarrow SmCor(Y_{\bullet,\bullet},R)$ that sends $R_{tr}(U,i)$ to $R_{tr}(U,i,i)$. As usual, this functor induces a pair of adjoint functors
 \begin{align*}
 	PST&(Y_{\bullet,\bullet},R)\\
 	d_{\#} \uparrow & \downarrow d^*\\
 	PST&(d(Y_{\bullet,\bullet}),R).
 \end{align*}
 
 As in \cite[Section 3]{voevodsky.simplicial}, we can define the tensor product of presheaves with transfers $F$ and $G$ on $Y_{\bullet,\bullet}$ in the following way
 $$(F \otimes G)_{i,h} = (F_{i,h} \otimes G_{i,h}) = h_0(Lres(F_{i,h}) \otimes Lres(G_{i,h} )).$$
 Let $D(Y_{\bullet,\bullet},R)$ be the derived category of complexes on $PST(Y_{\bullet,\bullet},R)$ bounded from above. Then, the tensor product just defined induces a tensor triangulated structure on $D(Y_{\bullet,\bullet},R)$ given by
 $$K \stackrel{L}{\otimes} L = Lres(K) \otimes Lres(L)$$
 for all complexes of presheaves with transfers $K$ and $L$.
 
 The unit of this tensor structure is the constant presheaf with transfers denoted also by $R$ whose components are the constant presheaves with transfers on each $Y_{i,h}$.
 
 \begin{lem}\label{tens}
 	Consider the bisimplicial object $LR_{\bullet,\bullet}$ in $SmCor(Y_{\bullet,\bullet},R)$ with terms 
 	$$LR_{i,h} =R_{tr}(Y_{i,h},i,h)$$
 	and the obvious structure morphisms. Let $LR_*$ be the total complex of the corresponding double complex. Then there is a natural quasi-isomorphism
 	$$LR_* \rightarrow R.$$
 \end{lem}
 \begin{proof}
 	See \cite[Lemma 3.9]{voevodsky.simplicial}.
 \end{proof}
 
 Let $W_{i,h}^{el}(Y_{\bullet,\bullet},R)$ be the class of complexes on $PST(Y_{\bullet,\bullet},R)$ obtained as $r_{i,h,\#}(W^{el}(Y_{i,h},R))$, where $W^{el}(Y_{i,h},R)$ is defined in \cite[Section 4]{voevodsky.simplicial}.
 
 Let $W(Y_{\bullet,\bullet},R)$ be the smallest localizing subcategory of $D(Y_{\bullet,\bullet},R)$ containing all $W_{i,h}^{el}(Y_{\bullet,\bullet},R)$. A morphism in $D(Y_{\bullet,\bullet},R)$ is called an $A^1$-equivalence if its cone lives in $W(Y_{\bullet,\bullet},R)$.
 
 \begin{dfn}
 	The triangulated category $\DM^-_{eff}(Y_{\bullet,\bullet},R)$ of motives over $Y_{\bullet,\bullet}$ is the localization of $D(Y_{\bullet,\bullet},R)$ with respect to $A^1$-equivalences.
 \end{dfn}
 
 What follows consists of a bunch of properties of the restriction functors whose simplicial analogues can be found in \cite[Sections 3 and 4]{voevodsky.simplicial}.
 
 The family $\{r_{i,h}^*\}_{i,h \geq 0}$ induces a family of restriction functors from $D(Y_{\bullet,\bullet},R)$ to $D(Y_{i,h},R)$, with respective left adjoints $Lr_{i,h,{\#}}$, that respect $A^1$-equivalences. Hence, we get a family of restriction functors $\{r_{i,h}^*\}_{i,h \geq 0}$ from $\DM_{eff}^-(Y_{\bullet,\bullet},R)$ to $\DM_{eff}^-(Y_{i,h},R)$ that is moreover conservative. The same is true also for the families of functors $\{r_{i,\bullet}^*\}_{i \geq 0}$ and $\{r_{\bullet,h}^*\}_{h \geq 0}$.
 
 The diagonal functor $d^*$ also induces a functor from $D(Y_{\bullet,\bullet},R)$ to $D(d(Y_{\bullet,\bullet}),R)$, with left adjoint $Ld_{\#}$, respecting $A^1$-equivalences. Therefore, we get a diagonal restriction functor $d^*$ from $\DM_{eff}^-(Y_{\bullet,\bullet})$ to $\DM_{eff}^-(d(Y_{\bullet,\bullet}))$.
 
 The tensor product on $D(Y_{\bullet,\bullet},R)$ respects $A^1$-equivalences, making $\DM^-_{eff}(Y_{\bullet,\bullet},R)$ into a tensor triangulated category. All the restriction functors introduced above respect this tensor structure.
 
 There are also standard functoriality properties. If $p:Y_{\bullet,\bullet} \rightarrow Y'_{\bullet,\bullet}$ is a morphism of smooth bisimplicial schemes, then we get a pair of adjoint functors
 \begin{align*}
 	\DM_{eff}^-&(Y_{\bullet,\bullet},R)\\
 	Lp^* \uparrow & \downarrow Rp_*\\
 	\DM_{eff}^-&(Y'_{\bullet,\bullet},R).
 \end{align*}
 If $p$ is smooth, then $Lp^*=p^*$ and there is also the following adjunction
 \begin{align*}
 	\DM_{eff}^-&(Y_{\bullet,\bullet},R)\\
 	Lp_{\#} \downarrow & \uparrow p^*\\
 	\DM_{eff}^-&(Y'_{\bullet,\bullet},R).
 \end{align*}
 In particular, we have a pair of adjoint functors
 \begin{align*}
 	\DM_{eff}^-&(Y_{\bullet,\bullet},R)\\
 	Lc_{\#} \downarrow & \uparrow c^*\\
 	\DM_{eff}^-&(k,R)
 \end{align*}
 where $c:Y_{\bullet,\bullet} \rightarrow Spec(k)$ is the projection to the base. 
 
 \begin{dfn}
 	A Tate object $T(q)[p]$ in $\DM^-_{eff}(Y_{\bullet,\bullet},R)$ is defined as $c^*(T(q)[p])$. In general, for any motive $M$ in $\DM^-_{eff}(k,R)$, we also denote by $M$ its image $c^*(M)$ in $\DM^-_{eff}(Y_{\bullet,\bullet},R)$.
 \end{dfn}
 
 \begin{rem}\label{tcomp}
 	\normalfont
 	Note that by \cite[Theorem 11.1.13]{cisinski.deglise} (which holds true over any diagram of schemes) the unit object $T$ is compact (and so are all $T(q)[p]$) in $\DM^-_{eff}(Y_{\bullet,\bullet},R)$.
 \end{rem}
 
 A smooth bisimplicial scheme $Y_{\bullet,\bullet}$ induces a bisimplicial object $R_{tr}(Y_{\bullet,\bullet})$ in $SmCor(k,R)$. Then, one can define the motive $M(Y_{\bullet,\bullet})$ of $Y_{\bullet,\bullet}$ in $\DM^{-}_{eff}(k,R)$ as the total complex of the double complex $R_{tr}(Y_{*,*})$ associated to $R_{tr}(Y_{\bullet,\bullet})$. It is an immediate consequence of Lemma \ref{tens} that $$M(Y_{\bullet,\bullet}) \cong Lc_{\#}T.$$
 The latter definition naturally extends to the bisimplicial case the definition of the motive of a simplicial scheme given in \cite[Section 5]{voevodsky.simplicial}. Note that, by the Eilenberg-Zilber theorem, $M(Y_{\bullet,\bullet})$ and $M(d(Y_{\bullet,\bullet}))$ are isomorphic in $\DM^{-}_{eff}(k,R)$. In particular, they have the same motivic cohomology.
 
 The most important result that we need in the following sections is the following.
 
 \begin{prop}
 	Let $Y_{\bullet,\bullet}$ be a smooth bisimplicial scheme. Then, there is an isomorphism
 	$$\Hom_{\DM^{-}_{eff}(Y_{\bullet,\bullet},R)}(T(q')[p'],T(q)[p]) \cong \Hom_{\DM^{-}_{eff}(k,R)}(M(Y_{\bullet,\bullet})(q')[p'],T(q)[p])$$
 	for all integers $q$, $q'$, $p$ and $p'$.
 \end{prop}
 \begin{proof}
 	See \cite[Proposition 5.3]{voevodsky.simplicial}.
 \end{proof}
 
 \section{A Serre spectral sequence for motivic cohomology} 
 
 The main purpose of this section is to construct Postnikov systems in a suitable triangulated category of motives and to study the associated spectral sequences. More precisely, we set our triangulated category $\mathcal{C}$ to be $\DM^-_{eff}(Y_{\bullet,\bullet},R)$, our abelian category $\A$ to be the category of left $H^{**}(Y_{\bullet,\bullet},R)$-modules and our cohomological functor $H$ to be motivic cohomology $H^{**}(-,R)$.
 
 For all $i \geq 0$ denote simply by
 $$r_i^*:\DM_{eff}^-(Y_{\bullet,\bullet},R) \rightarrow \DM_{eff}^-(Y_{i,\bullet},R)$$
 the restriction functors $r_{i,\bullet}^*$ introduced in the last section, and by $Lr_{i,\#}$ the respective left adjoint functors. The image of a motive $N$ in $\DM_{eff}^-(Y_{\bullet,\bullet},R)$ under $r_i^*$ is simply denoted by $N_i$.
 
 Now let us recall some facts about coherence from \cite{smirnov.vishik} and adapt them to the bisimplicial case we are interested in. 
 
 \begin{dfn} 
 	A smooth coherent morphism of smooth bisimplicial schemes is a smooth morphism $\pi:X_{\bullet,\bullet} \rightarrow Y_{\bullet,\bullet}$ such that there is a cartesian square of simplicial schemes
 	$$
 	\xymatrix{
 		X_{j,\bullet} \ar@{->}[r]^{\pi_j} \ar@{->}[d]_{X_ {\theta}} & Y_{j,\bullet}  \ar@{->}[d]^{Y_ {\theta} }\\
 		X_{i,\bullet}  \ar@{->}[r]_{\pi_{i}} & Y_{i,\bullet} 
 	}
 	$$
 	for any simplicial map $\theta:[i] \rightarrow [j]$.
 \end{dfn} 
 
 \begin{dfn}
 	A motive $N$ in $\DM_{eff}^-(Y_{\bullet,\bullet},R)$ is said to be coherent if all simplicial morphisms $\theta:[i] \rightarrow [j]$ induce structural isomorphisms $N_ \theta :LY_ {\theta} ^*(N_i) \rightarrow N_j$ in $\DM_{eff}^-(Y_{j,\bullet},R)$. The full subcategory of $\DM_{eff}^-(Y_{\bullet,\bullet},R)$ consisting of coherent motives is denoted by $\DM_{coh}^-(Y_{\bullet,\bullet},R)$.
 \end{dfn}
 
 \begin{rem}\label{cohloc}
 	\normalfont
 	Note that $\DM_{coh}^-(Y_{\bullet,\bullet},R)$ is a localizing subcategory of $\DM_{eff}^-(Y_{\bullet,\bullet},R)$. Since $L\pi_{\#}$ maps coherent motives to coherent ones for any smooth coherent morphism $\pi$, we have that $M(X_{\bullet,\bullet} \xrightarrow{\pi} Y_{\bullet,\bullet})$ is an object in $\DM_{coh}^-(Y_{\bullet,\bullet},R)$, where $M(X_{\bullet,\bullet} \xrightarrow{\pi} Y_{\bullet,\bullet})$ is the image $L\pi_{\#}(T)$ of the unit Tate motive in $\DM_{eff}^-(X_{\bullet,\bullet},R)$.
 \end{rem}
 
 \begin{prop}\label{filtr}
 	For any motive $N$ in $\DM_{coh}^-(Y_{\bullet,\bullet},R)$ there exists a functorial increasing filtration 
 	$$(N)_{\leq 0} \rightarrow (N)_{\leq 1} \rightarrow \dots \rightarrow (N)_{\leq n-1} \rightarrow (N)_{\leq n} \rightarrow \dots \rightarrow N$$
 	with graded pieces $(N)_n=Cone((N)_{\leq n-1} \rightarrow (N)_{\leq n}) \cong Lr_{n,\#}r^*_n(N)[n]$ which converges in the sense that
 	$$\bigoplus_n (N)_{\leq n} \xrightarrow{id-sh} \bigoplus_n (N)_{\leq n} \rightarrow N$$
 	extends to a distinguished triangle, where $sh:(N)_{\leq n-1} \rightarrow (N)_{\leq n}$ is the map from the filtration.
 \end{prop}
 \begin{proof}
 	The proofs of \cite[Propositions 3.1.6 and 3.1.8]{smirnov.vishik} extend verbatim to the bisimplicial case. 
 \end{proof}
 
 The next proposition is a generalisation of \cite[Proposition 3.1.5]{smirnov.vishik}. Indeed, it allows to construct Postnikov systems for coherent motives with simplicial components which are direct sums of Tate motives satisfying some specific conditions. The proof follows the guidelines of \cite[Proposition 3.1.5]{smirnov.vishik} and essentially reproduces the same arguments in our more general context. Before proceeding, we need to define a strict order relation on the bidegrees $(q)[p]$.
 
 \begin{dfn}\label{ord}
 	We set $(q)[p] \prec (q')[p']$ if and only if one of the following two conditions is satisfied:\\
 	1) $q<q'$;\\
 	2) $q=q'$ and $p<p'$.
 \end{dfn}
 
 For any $j \geq 0$, let $T^j$ be the possibly infinite direct sum $\bigoplus_{I_j} T(q_j)[p_j]$ in $\DM_{eff}^-(k,R)$ such that $(q_j)[p_j] \prec (q_{j+1})[p_{j+1}]$ and let $N$ in $\DM_{coh}^-(Y_{\bullet,\bullet},R)$ be such a motive that its simplicial components $N_i$ in $\DM_{eff}^-(Y_{i,\bullet},R)$ are isomorphic to the direct sum $\bigoplus_{j \geq 0} T^j$.
 
 Note that in $\DM_{eff}^-(Y_{1,\bullet},R)$ the automorphism group $Aut(\bigoplus_{j \geq 0} T^j)$ consists of invertible upper triangular matrices, since
 $$\Hom_{\DM_{eff}^-(Y_{1,\bullet},R)}(T,T(q)[p]) \cong H^{p,q}(Y_{1,\bullet},R) \cong 0$$ 
 for any $(q)[p] \prec (0)[0]$.
 
 Recall that, since $N$ is coherent, for any simplicial map $\theta:[i] \rightarrow [j]$, the structural map $N_{\theta}:LY^*_{\theta}(N_i) \rightarrow N_j$ is an isomorphism. In particular, we have two automorphisms
 $$N_{\partial_0}:LY^*_{\partial_0}(N_0) \cong \bigoplus_{j \geq 0} T^j \rightarrow N_1 \cong \bigoplus_{j \geq 0} T^j$$ 
 and
 $$N_{\partial_1}:LY^*_{\partial_1}(N_0) \cong \bigoplus_{j \geq 0} T^j \rightarrow N_1 \cong \bigoplus_{j \geq 0} T^j.$$ 
 
 \begin{dfn}
 	Denote by $\omega^N$ the composition $N_{\partial_1} \circ N_{\partial_0}^{-1}$ belonging to $Aut(N_1)$. For any $j \geq 0$, consider the homomorphism $Aut(\bigoplus_{j \geq 0} T^j) \rightarrow Aut(T^j)$ that sends each invertible upper triangular matrix to its $j$th diagonal entry. We denote the images of $\omega^N$ under these homomorphisms by $\omega^N_j$.
 \end{dfn}
 
 \begin{rem}
 	\normalfont
 	Suppose that $I_j$ is a finite set of order $n_j$ and let $CC(Y_{\bullet,\bullet})$ be the bisimplicial set obtained by applying to $Y_{\bullet,\bullet}$ the connected components functor $CC$ that sends a connected scheme to the point and respects coproducts. Denote by $|CC(Y_{\bullet,\bullet})|$ its geometric realization.
 	
 	If the nonabelian cohomology $H^1(|CC(Y_{\bullet,\bullet})|,GL_{n_j}(R))$ is trivial, then $\omega^N_j$ is the identity.
 	
 	In fact, $\omega^N_j$ is an automorphism of $T^j$ in $\DM_{eff}^-(Y_{1,\bullet},R)$, thus an invertible element of 
 	$$\Hom_{\DM_{eff}^-(Y_{1,\bullet},R)}(T^j,T^j) \cong \prod_{I_j} \bigoplus_{I_j}  \Hom_{\DM_{eff}^-(Y_{1,\bullet},R)}(T,T)\cong \prod_{I_j} \bigoplus_{I_j} H^{0,0}(Y_{1,\bullet},R)$$
 	by Remark \ref{tcomp}.
 	
 	Note that $H^{0,0}(Y_{1,\bullet},R)$ is the free $R$-module of rank given by the number of connected components of $Y_{1,\bullet}$. Hence, if we fix a connected component of $Y_{1,\bullet}$, $\omega^N_j$ restricts to an invertible element of
 	$$\prod_{I_j} \bigoplus_{I_j} R \cong M_{n_j}(R),$$
 	i.e. to an element of $GL_{n_j}(R)$.
 	This way one can see $\omega^N_j$ as a morphism of groupoids
 	$$\Pi_1(|CC(Y_{\bullet,\bullet})|) \rightarrow GL_{n_j}(R),$$
 	so as an element of $H^1(|CC(Y_{\bullet,\bullet})|,GL_{n_j}(R))$.
 \end{rem}
 
 \begin{prop}\label{post}
 	
 	Let $N$ be a motive in $\DM_{coh}^-(Y_{\bullet,\bullet},R)$ such that its simplicial components $N_i$ in $\DM_{eff}^-(Y_{i,\bullet},R)$ are isomorphic to the direct sum $\bigoplus_{j \geq 0} T^j$, where $T^j$ is the possibly infinite direct sum $\bigoplus_{I_j} T(q_j)[p_j]$ such that $(q_j)[p_j] \prec (q_{j+1})[p_{j+1}]$ for all $j$.
 	
 	If $\omega^N_j$ is trivial for any $j \geq 0$, then there exists a Postnikov system in $\DM_{eff}^-(Y_{\bullet,\bullet},R)$
 	$$
 	\xymatrix{
 		\dots \ar@{->}[r] & N^{j+1} \ar@{->}[r] \ar@{->}[d]	 &N^j \ar@{->}[r] \ar@{->}[d] &  \dots \ar@{->}[r]   & N^1 \ar@{->}[r] \ar@{->}[d]	 &N=N^0 \ar@{->}[d] \\
 		& T^{j+1} \ar@{->}[ul]^{[1]}  &	T^j \ar@{->}[ul]^{[1]} & & T^1 \ar@{->}[ul]^{[1]}  &	T^0 \ar@{->}[ul]^{[1]} & 
 	}
 	$$
 	such that the simplicial components $N^j_i$ are isomorphic to the direct sum $\bigoplus_{k \geq j} T^k$ and the morphisms $r_i^*(N^j \rightarrow T^j)$ are the natural projections $\bigoplus_{k \geq j} T^k \rightarrow T^j$ in $\DM_{eff}^-(Y_{i,\bullet},R)$.
 \end{prop}
 \begin{proof}
 	To construct the aimed Postnikov system we just need to produce morphisms $N^j \rightarrow T^j$ where each $N^j$ is defined as the cone of the previous morphism, namely $N^j=Cone(N^{j-1} \rightarrow T^{j-1})[-1]$. We proceed by induction.
 	
 	Notice that each simplicial component of $N$ is isomorphic to $\bigoplus_{j \geq 0} T^j$ and $T^0$ is the direct sum of possibly infinite $T(q_0)[p_0]$ such that $(q_0)[p_0] \prec (q_j)[p_j]$ for any $j \geq 1$ by hypothesis. By applying the triangulated functor $Lc_{\#}$ to the filtration of Proposition \ref{filtr}, one gets a filtration $(Lc_{\#}N)_{\leq n}$ for $Lc_{\#}N$ with graded pieces $(Lc_{\#}N)_n \cong \bigoplus_{j \geq 0} \bigoplus_{I_j} M(Y_{n,\bullet})(q_j)[p_j+n]$. Following the lines of the proof of \cite[Proposition 3.1.5]{smirnov.vishik}, we denote by $(Lc_{\#}N)_{>n}$ the cone $Cone((Lc_{\#}N)_{\leq n} \rightarrow Lc_{\#}N)$ and by $(Lc_{\#}N)_{m \geq *>n}$ the cone $Cone((Lc_{\#}N)_{\leq n} \rightarrow (Lc_{\#}N)_{\leq m})$ for any $m>n$. Now, note that 
 	$$(Lc_{\#}N)_{>n} \cong Cone(\bigoplus_{m>n}(Lc_{\#}N)_{m \geq *>n} \xrightarrow{id-sh} (Lc_{\#}N)_{m \geq *>n} )$$ 
 	and moreover $(Lc_{\#}N)_{m \geq *>n}$ is an extension of $(Lc_{\#}N)_k$ for $n<k \leq m$. Therefore, we have that
 	$$\Hom_{\DM^{-}_{eff}(k,R)}((Lc_{\#}N)_{>0},T^0) \cong 0,$$
 	$$\Hom_{\DM^{-}_{eff}(k,R)}((Lc_{\#}N)_{>1},T^0) \cong 0,$$
 	$$\Hom_{\DM^{-}_{eff}(k,R)}((Lc_{\#}N)_{>1},T^0[1]) \cong 0,$$
 	since $\Hom_{\DM^{-}_{eff}(k,R)}(M(Y_{n,\bullet}),T(q)[p]) \cong 0$ for any $n \geq 0$ and any $(q)[p] \prec (0)[0]$. We deduce from these remarks and by applying the cohomological functor $\Hom_{\DM^{-}_{eff}(k,R)}(-,T^0)$ to the distinguished triangle
 	$$(Lc_{\#}N)_0 \rightarrow Lc_{\#}N \rightarrow (Lc_{\#}N)_{>0} \rightarrow (Lc_{\#}N)_0[1]$$
 	that there exists an exact sequence
 	$$0 \rightarrow \Hom_{\DM^{-}_{eff}(k,R)}(Lc_{\#}(N),T^0) \rightarrow \Hom_{\DM^{-}_{eff}(k,R)}(Lc_{\#,0}(N_0),T^0)$$
 	$$\rightarrow \Hom_{\DM^{-}_{eff}(k,R)}(Lc_{\#,1}(N_1),T^0).$$
 	Repeating the same arguments for $T^0$ in $\DM_{coh}^-(Y_{\bullet,\bullet},R)$ one gets a similar sequence 
 	$$0 \rightarrow \Hom_{\DM^{-}_{eff}(k,R)}(Lc_{\#}(T^0),T^0) \rightarrow \Hom_{\DM^{-}_{eff}(k,R)}(Lc_{\#,0}(T^0),T^0)$$
 	$$\rightarrow \Hom_{\DM^{-}_{eff}(k,R)}(Lc_{\#,1}(T^0),T^0).$$
 	At this point we want to produce an isomorphism between $\Hom_{\DM^{-}_{eff}(k,R)}(Lc_{\#}(N),T^0)$ and\\
 	$\Hom_{\DM^{-}_{eff}(k,R)}(Lc_{\#}(T^0),T^0)$. In order to do so, we need to identify the last morphisms of the two exact sequences. Since in the exact sequences only the $0$th and the $1$st simplicial components appear, it is enough to get a compatibility between the coherent system $(N_i,N_{\theta})$ and the one of $T^0$ for $i=0,1$ and simplicial maps $\theta:[0] \rightarrow [1]$, where $N_{\theta}$ is the structural isomorphism $LY^*_{\theta}(N_0) \rightarrow N_1$ in $\DM_{eff}^-(Y_{1,\bullet},R)$. In other words, we want a commutative diagram
 	$$
 	\xymatrix{
 		N_1 \cong \bigoplus_{j \geq 0} T^j \ar@{->}[r]^{\omega^N} \ar@{->}[d] & N_1 \cong \bigoplus_{j \geq 0} T^j \ar@{->}[d]\\
 		T^0  \ar@{->}[r]_{id} & T^0.
 	}
 	$$
 	Indeed, we have such a commutative diagram since by hypothesis $\omega^N_0$ is trivial. Hence, the two exact sequences above coincide. Then, we have that
 	$$\Hom_{\DM^{-}_{eff}(Y_{\bullet,\bullet},R)}(N,T^0) \cong \Hom_{\DM^{-}_{eff}(k,R)}(Lc_{\#}(N),T^0) \cong$$
 	$$\Hom_{\DM^{-}_{eff}(k,R)}(Lc_{\#}(T^0),T^0) \cong \Hom_{\DM^{-}_{eff}(Y_{\bullet,\bullet},R)}(T^0,T^0)$$
 	by adjunctions and the identity of $T^0$ provides the pursued morphism $N \rightarrow T^0$ whose restriction on each simplicial component is given by the natural projection $\bigoplus_{k \geq 0}T^k \rightarrow T^0$. It follows that $N^1_i$ is isomorphic to $\bigoplus_{k \geq 1}T^k$ for any $i$. This proves the induction basis.
 	
 	Now, suppose we have a morphism from $N^k$ to $T^k$ for any $0 \leq k \leq j-1$, where each $N^k$ is defined as $Cone(N^{k-1} \rightarrow T^{k-1})[-1]$. We denote by $N^j$ the cone $Cone(N^{j-1} \rightarrow T^{j-1})[-1]$. Notice that the simplicial components of $N^j$ are all isomorphic to $\bigoplus_{l \geq j} T^l$ and $T^j$ is the direct sum of possibly infinite $T(q_j)[p_j]$ such that $(q_j)[p_j] \prec (q_l)[p_l]$ for any $l \geq j+1$ by hypothesis. Therefore, by applying the same arguments of the induction basis to $N^j$, using the fact that $\omega^{N^j}_0=\omega^N_j$ is trivial by hypothesis, there exists a morphism $N^j \rightarrow T^j$. This completes the proof.
 \end{proof}
 
 \begin{rem}
 	\normalfont
 	We point out that, in the case it exists an integer $k$ such that $I_j$ is empty for all $j \geq k$, the previous result provides a finite Postnikov system for $N$ in $\DM_{eff}^-(Y_{\bullet,\bullet},R)$.
 \end{rem}
 
 We want to apply Proposition \ref{post} to produce a spectral sequence for morphisms having motivically cellular fibers, i.e. fibers whose motives are direct sums of Tate motives satisfying certain conditions. First, we need to construct suitable Postnikov systems. The next result is a generalisation of \cite[Proposition 4.2]{tanania.b}.
 
 \begin{prop} \label{Serre}
 	Let $\pi:X_{\bullet,\bullet} \rightarrow Y_{\bullet,\bullet}$ be a smooth coherent morphism of smooth bisimplicial schemes over $k$, $A_{\bullet}$ a smooth simplicial $k$-scheme and, for any $j \geq 0$, $T^j$ the possibly infinite direct sum of Tate motives $\bigoplus_{I_j} T(q_j)[p_j]$ in $\DM_{eff}^-(k,R)$ such that $(q_j)[p_j] \prec (q_{j+1})[p_{j+1}]$. Moreover, suppose the following conditions hold:\\
 	1) over the $0$th simplicial component $\pi$ is isomorphic to the projection $Y_{0,\bullet} \times A_{\bullet} \rightarrow Y_{0,\bullet}$;\\
 	2) $\omega^{M(X_{\bullet,\bullet} \rightarrow Y_{\bullet,\bullet})}_j$ is trivial for any $j \geq 0$;\\
 	3) $M(A_{\bullet}) \cong \bigoplus_{j \geq 0} T^j \in \DM_{eff}^-(k,R)$.\\
 	Then, there exists a Postnikov system in $\DM_{eff}^-(Y_{\bullet,\bullet},R)$
 	$$
 	\xymatrix{
 		\dots \ar@{->}[r] & N^{j+1} \ar@{->}[r] \ar@{->}[d]	 &N^j \ar@{->}[r] \ar@{->}[d] &  \dots \ar@{->}[r]   & N^2 \ar@{->}[r] \ar@{->}[d]	 &N^1 \ar@{->}[r] \ar@{->}[d]  &M(X_{\bullet,\bullet} \xrightarrow{\pi} Y_{\bullet,\bullet})=N^0 \ar@{->}[d]\\
 		& T^{j+1} \ar@{->}[ul]^{[1]}  &	T^j \ar@{->}[ul]^{[1]} & & T^2 \ar@{->}[ul]^{[1]}  &	T^1 \ar@{->}[ul]^{[1]} & T^0 \ar@{->}[ul]^{[1]}
 	}
 	$$
 	such that the simplicial components $N^j_i$ are isomorphic to the direct sum $\bigoplus_{k \geq j} T^k$ and the morphisms $r_i^*(N^j \rightarrow T^j)$ are the natural projections $\bigoplus_{k \geq j} T^k \rightarrow T^j$ in $\DM_{eff}^-(Y_{i,\bullet},R)$.
 \end{prop}
 \begin{proof}
 	By coherence of $\pi$, we have that $\pi_i:Y_{i,\bullet} \times A_{\bullet} \cong X_{i,\bullet} \rightarrow Y_{i,\bullet}$ is the projection onto the first factor for any $i$. It follows that the coherent motive $N^0$ (see Remark \ref{cohloc}) has simplicial components given by $N^0_i \cong M(A_{\bullet})$ in $\DM_{eff}^-(Y_{i,\bullet},R)$ for any $i$. Therefore, Proposition \ref{post} implies the existence of the aimed Postnikov system in $\DM_{eff}^-(Y_{\bullet,\bullet},R)$, and the proof is complete. 
 \end{proof} 
 
 Recall from Section \ref{spec} that, once constructed a Postnikov system in a triangulated category and considered a suitable cohomological functor, one can obtain a spectral sequence which may converge if some extra requirements are met. The following theorem just states the existence of a strongly convergent spectral sequence related to the Postnikov system of Proposition \ref{Serre}.
 
 \begin{thm}\label{Serre2}
 	Let $\pi:X_{\bullet,\bullet} \rightarrow Y_{\bullet,\bullet}$ be a smooth coherent morphism of smooth bisimplicial schemes over $k$ and $A_{\bullet}$ a smooth simplicial $k$-scheme satisfying all conditions of Proposition \ref{Serre}. Moreover, for any bidegree $(q)[p]$, suppose there is an integer $l$ such that $(q)[p] \prec (q_l)[p_l]$. Then, there exists a strongly convergent spectral sequence
 	$$E_1^{p,q,s}=\prod_{I_s} H^{p-p_s,q-q_s}(Y_{\bullet,\bullet},R) \Longrightarrow H^{p,q}(X_{\bullet,\bullet},R).$$
 \end{thm}
 \begin{proof}
 	We start by applying the construction of the exact couple associated to a Postnikov system of Section \ref{spec} to the cohomological functor $\Hom_{\DM_{eff}^-(Y_{\bullet,\bullet},R)}(-,T(q))$, for any $q$. This way, we get a spectral sequence with $E_1$-page given by 
 	$$E_1^{p,q,s}=\Hom_{\DM_{eff}^-(Y_{\bullet,\bullet},R)}(T^s,T(q)[p]) \cong \prod_{I_s} H^{p-p_s,q-q_s}(Y_{\bullet,\bullet},R).$$
 	The filtration we are considering is defined by $F^m=\ker(H^{**}(X_{\bullet,\bullet},R) \rightarrow H^{**}(N^m,R))$. In order to get the strong convergence we need to check that $\varinjlim_m H^{**}(N^m,R) \cong 0$. Since all the $N^m$ are coherent motives, by Proposition \ref{filtr} we have filtrations $(N^m)_{\leq n}$ with graded pieces $(N^m)_n \cong Lr_{n,\#}r^*_n(N^m)[n]$. Hence, we have filtrations $(Lc_{\#}N^m)_{\leq n}$ with graded pieces
 	$$(Lc_{\#}N^m)_n \cong \bigoplus_{k \geq m} \bigoplus_{I_k} M(Y_{n,\bullet})(q_k)[p_k+n].$$ 
 	Now, fix a bidegree $(q)[p]$, then by hypothesis there exists an integer $l$ such that $(q)[p] \prec (q_l)[p_l]$, from which it follows that
 	$$\Hom_{\DM_{eff}^-(k,R)}((Lc_{\#}N^l)_n,T(q)[p]) \cong 0$$
 	for any $n$. Therefore,
 	$$\Hom_{\DM_{eff}^-(k,R)}(Lc_{\#}(N^l),T(q)[p]) \cong 0$$ 
 	from which we deduce by adjunction that $H^{p,q}(N^l,R) \cong 0$ that implies, in particular, the triviality of $\varinjlim_m H^{**}(N^m,R)$. Hence, by Theorem \ref{SC} we obtain the result.
 \end{proof}
 
 The next result assures that the spectral sequence just constructed is functorial.
 
 \begin{prop}\label{Serre 3}
 	Let $\pi:X_{\bullet,\bullet} \rightarrow Y_{\bullet,\bullet}$ and $\pi':X'_{\bullet,\bullet} \rightarrow Y'_{\bullet,\bullet}$ be smooth coherent morphisms of smooth bisimplicial schemes over $k$ and $A_{\bullet}$ a smooth simplicial $k$-scheme that satisfies all conditions from Proposition \ref{Serre} with respect to $\pi'$ and such that the following square is cartesian with all morphisms smooth
 	$$
 	\xymatrix{
 		X_{\bullet,\bullet} \ar@{->}[r]^{\pi} \ar@{->}[d]_{p_X} & Y_{\bullet,\bullet} \ar@{->}[d]^{p_Y}\\
 		X'_{\bullet,\bullet} \ar@{->}[r]_{\pi'} & Y'_{\bullet,\bullet}
 	}
 	$$
 	Then, the induced morphism $Lp_{Y\#}M(X_{\bullet,\bullet}\xrightarrow{\pi}Y_{\bullet,\bullet}) \rightarrow M(X'_{\bullet,\bullet}\xrightarrow{\pi'}Y'_{\bullet,\bullet})$ in $\DM_{eff}^-(Y'_{\bullet,\bullet},R)$ extends uniquely to a morphism of Postnikov systems where, for any $j \geq 0$, $Lp_{Y\#}T^j \rightarrow T^j$ is given by $\bigoplus_{I_j} M(p_Y)(q_j)[p_j]$.
 \end{prop}
 \begin{proof}
 	We denote by $N^j$ the objects from the Postnikov system of $\pi$ and by $N'^j$ the ones from the Postnikov system of $\pi'$. 
 	
 	First, recall that, by Proposition \ref{filtr}, there is a filtration of $Lc_{\#}N^j$ with graded pieces 
 	$$(Lc_{\#}N^j)_n \cong \bigoplus_{k \geq j} \bigoplus_{I_k} M(Y_{n,\bullet})(q_k)[p_k+n].$$ 
 	It follows that $\Hom_{\DM_{eff}^-(k,R)}((Lc_{\#}N^j)_n,T^{j-1}[-1]) \cong 0$ and $\Hom_{\DM_{eff}^-(k,R)}((Lc_{\#}N^j)_n,T^{j-1}) \cong 0$ for any $n$ since, for any $k \geq j$, we have that $(q_{j-1})[p_{j-1}] \prec (q_k)[p_k]$ by hypothesis. Therefore, 
 	$$\Hom_{\DM_{eff}^-(Y'_{\bullet,\bullet},R)}(Lp_{Y\#}N^j,T^{j-1}[-1]) \cong \Hom_{\DM_{eff}^-(Y_{\bullet,\bullet},R)}(N^j,T^{j-1}[-1]) \cong$$
 	$$\Hom_{\DM_{eff}^-(k,R)}(Lc_{\#}N^j,T^{j-1}[-1]) \cong 0$$
 	and, similarly,
 	$$\Hom_{\DM_{eff}^-(Y'_{\bullet,\bullet},R)}(Lp_{Y\#}N^j,T^{j-1}) \cong \Hom_{\DM_{eff}^-(Y_{\bullet,\bullet},R)}(N^j,T^{j-1}) \cong$$
 	$$\Hom_{\DM_{eff}^-(k,R)}(Lc_{\#}N^j,T^{j-1}) \cong 0$$
 	from which we deduce that there are no non-trivial morphisms from $Lp_{Y\#}N^j $ to either $T^{j-1}[-1]$ or $T^{j-1}$ for any $j$.
 	
 	Now, we can construct the morphism of Postnikov systems by induction on $j$. The induction basis is provided by the square of motives in $\DM_{eff}^-(Y'_{\bullet,\bullet},R)$ induced by the geometric square of the hypothesis. Suppose by induction hypothesis that there is a morphism $Lp_{Y\#}N^{j-1} \rightarrow N'^{j-1}$. It follows that there exist unique morphisms $Lp_{Y\#}T^{j-1} \rightarrow T^{j-1}$ and $Lp_{Y\#}N^j \rightarrow N'^j$ fitting into a morphism of Postnikov systems in $\DM_{eff}^-(Y'_{\bullet,\bullet},R)$
 	$$
 	\xymatrix{
 		\dots \ar@{->}[r] &Lp_{Y\#}N^j \ar@{->}[rr] \ar@{->}[dd]& & Lp_{Y\#}N^{j-1} \ar@{->}[ld] \ar@{->}[dd] \ar@{->}[r] & \dots\\
 		& & Lp_{Y\#}T^{j-1} \ar@{->}[ul]^{[1]} \ar@{->}[dd]& &\\
 		\dots \ar@{->}[r] & N'^j \ar@{->}[rr] & & N'^{j-1} \ar@{->}[ld] \ar@{->}[r] & \dots\\
 		& & T^{j-1} \ar@{->}[ul]^{[1]}& &
 	}
 	$$
 	If we restrict our previous diagram to the $0$th simplicial component we obtain in $\DM_{eff}^-(Y'_{0,\bullet},R)$ the following morphism of Postnikov systems
 	$$
 	\xymatrix{
 		\dots \ar@{->}[r] &\bigoplus_{k \geq j}Lp_{Y_0\#}T^k \ar@{->}[rr] \ar@{->}[dd]& & \bigoplus_{k \geq j-1}Lp_{Y_0\#}T^k \ar@{->}[ld] \ar@{->}[dd] \ar@{->}[r] & \dots\\
 		& & Lp_{Y_0\#}T^{j-1} \ar@{->}[ul]^{[1]} \ar@{->}[dd]& &\\
 		\dots \ar@{->}[r] & \bigoplus_{k \geq j} T^k \ar@{->}[rr] & & \bigoplus_{k \geq j-1} T^k \ar@{->}[ld] \ar@{->}[r] & \dots\\
 		& & T^{j-1} \ar@{->}[ul]^{[1]}& &
 	}
 	$$
 	where each triangle is split. By hypothesis, the morphism $Lp_{Y_0\#}T^{j-1} \rightarrow T^{j-1}$ in the previous diagram is basically given by $\bigoplus_{I_{j-1}} M(p_{Y_0})(q_{j-1})[p_{j-1}]$, while the map $Lp_{Y_0\#}N_0^{j-1} \rightarrow N_0'^{j-1}$ is given by $\bigoplus_{k \geq {j-1}}\bigoplus_{I_k} M(p_{Y_0})(q_k)[p_k]$.
 	
 	Now, note that by Remark \ref{tcomp} the morphisms $Lp_{Y\#}T^j \rightarrow T^j$ and $\bigoplus_{I_j} M(p_Y)(q_j)[p_j]$ are both in 
 	$$\Hom_{\DM_{eff}^-(Y'_{\bullet,\bullet},R)}(Lp_{Y\#}T^j,T^j) \cong \Hom_{\DM_{eff}^-(Y_{\bullet,\bullet},R)}(T^j,p_Y^*T^j) \cong$$
 	$$\Hom_{\DM_{eff}^-(Y_{\bullet,\bullet},R)}(T^j,T^j) \cong \prod_{I_j} \bigoplus_{I_j} H^{0,0}(Y_{\bullet,\bullet},R)$$
 	and, for the same reason, $(Lp_{Y_0\#}T^j \rightarrow T^j)=\bigoplus_{I_j} M(p_{Y_0})(q_j)[p_j]$ is in 
 	$$\Hom_{\DM_{eff}^-(Y'_{0,\bullet},R)}(Lp_{Y_0\#}T^j,T^j) \cong \Hom_{\DM_{eff}^-(Y_{0,\bullet},R)}(T^j,p_{Y_0}^*T^j)
 	\cong$$
 	$$\Hom_{\DM_{eff}^-(Y_{0,\bullet},R)}(T^j,T^j) \cong \prod_{I_j} \bigoplus_{I_j} H^{0,0}(Y_{0,\bullet},R).$$
 	Recall that $H^{0,0}(Y_{\bullet,\bullet},R)$ is the free $R$-module with rank equal to the number of connected components of $Y_{\bullet,\bullet}$ and, analogously, $H^{0,0}(Y_{0,\bullet},R)$ is the free $R$-module with rank equal to the number of connected components of $Y_{0,\bullet}$. Since, as in the argument at the end of \cite[Proposition 3.4]{tanania.a}, the homomorphism
 	$$r_0^*:H^{0,0}(Y_{\bullet,\bullet},R) \rightarrow  H^{0,0}(Y_{0,\bullet},R)$$
 	is injective, we deduce that $Lp_{Y\#}T^j \rightarrow T^j$ and $\bigoplus_{I_j} M(p_Y)(q_j)[p_j]$ are identified, which completes the proof.
 \end{proof}
 
 \begin{rem}\label{csta}
 	\normalfont
 	Note that we can always ``dilute" the Postnikov system of Proposition \ref{Serre} by allowing some empty sets $I_j$. In particular, if we are in the situation of Proposition \ref{Serre 3}, but without assuming that the square is cartesian (i.e. $\pi$ and $\pi'$ both satisfy the conditions of Proposition \ref{Serre} with possibly different fibers $A_{\bullet}$ and $A'_{\bullet}$ respectively), we can stretch both the Postnikov systems for $N$ and $N'$ so that they have the same set of weights (take for example the union of the two sets of weights). Then, by the same proof of the previous proposition, we still have a unique morphism of Postnikov systems. The only thing we lose, which indeed requires the square to be cartesian, is the description of the morphisms on the slices $Lp_{Y\#}T^j \rightarrow T'^j$. 
 \end{rem}
 
 We would like to finish this section by establishing a comparison between the spectral sequence here presented and the Serre spectral sequence associated to a fiber bundle in topology. Recall that in topology for a fibre sequence
 $$F \rightarrow E \rightarrow B$$
 with $\pi_1(B)$ acting trivially on $H^*(F)$ one has a spectral sequence converging to $H^*(E)$
 $$E^{s,t}_2=H^s(B,H^t(F)) \Longrightarrow H^*(E)$$
 called Serre spectral sequence (see for example \cite[Theorem 15.27]{switzer}).
 
 Analogously, our spectral sequence allows to reconstruct somehow the cohomology of the total bisimplicial scheme from the cohomology of the base and of the fiber, provided that the fiber is motivically cellular. Moreover, the triviality condition on the $\omega^{M(X_{\bullet,\bullet} \rightarrow Y_{\bullet,\bullet})}_j$ for any $j \geq 0$ is reminiscent of the topological condition on the triviality of the action of $\pi_1(B)$ on $H^*(F)$. On the other hand, the main difference between the two spectral sequences resides in how they are obtained. In fact, while the topological Serre spectral sequence is classically achieved by filtering the base, our spectral sequence is instead realized by filtering the fiber.
 
 \section{The multiplicative structure}
 
 We are now ready to discuss the multiplicative properties of the motivic Serre spectral sequence constructed in the previous section.
 
 \begin{dfn}
 	A set of bidegrees $\{(q_j)[p_j]\}_{j\geq0}$, ordered as in Definition \ref{ord}, is called collinear if $p_iq_j=q_ip_j$ for all $i,j \geq 0$. A motive $\bigoplus_{j \geq 0} T^j$ is called collinearly weighted if its set of bidegrees is collinear.
 \end{dfn}
 
 Note that, if $A_{\bullet}$ is a simplicial scheme whose motive is $M(A_{\bullet}) \cong \bigoplus_{j \geq 0} T^j$, then $p_j \geq q_j \geq 0$ for all $j \geq 0$ (with $p_0=q_0=0$). Hence, if $M(A_{\bullet})$ is collinearly weighted, its set of bidegrees is contained in a maximal collinear set of type $\{(qk)[pk]\}_{k \geq0}$ for some relatively prime $p \geq q \geq 0$.
 
 Let $\pi:X_{\bullet,\bullet} \rightarrow Y_{\bullet,\bullet}$ be a map satisfying the conditions of Proposition \ref{Serre} with a collinearly weighted $M(A_{\bullet})$, and consider the associated diagonal map $\Delta: X_{\bullet,\bullet} \rightarrow X_{\bullet,\bullet}  \times_{Y_{\bullet,\bullet}} X_{\bullet,\bullet}$. Note that the sets of weights of $M(A_{\bullet})$ and $M(A_{\bullet} \times A_{\bullet})$ are both contained in the same maximal collinear set.
 
 Now, construct the Postnikov system for $N$, according to Proposition \ref{Serre}, with $M(A_{\bullet}) \cong \bigoplus_{k \geq 0} T^k$ where $\{(q_k)[p_k]\}_{k \geq 0}$ is the corresponding maximal collinear set of weights (we are allowing some $I_k$ to be empty). Then, consider the Postnikov system for $N \otimes N$ parametrized in the same way, provided by Proposition \ref{Serre},  whose slices are given by
 $$(N \otimes N)^k/(N \otimes N)^{k+1} = Cone((N \otimes N)^{k+1} \rightarrow (N \otimes N)^k) \cong \bigoplus_{i+j=k}T^i \otimes T^j.$$
 
 By Remark \ref{csta}, the morphism $N \rightarrow N \otimes N$, induced by $\Delta$, extends uniquely to a morphism of Postnikov systems $N^k \rightarrow (N \otimes N)^k$ in $\DM_{eff}^-(Y_{\bullet,\bullet},R)$.
 
 At this point, we would like to produce some Cartan-Eilenberg systems out of these Postnikov systems in order to study the multiplicative structure of our spectral sequence. The standard reference for Cartan-Eilenberg systems and multiplicative structure of spectral sequences is \cite{douady}, but in this section we will mainly refer to \cite{rognes} where general definitions and full proofs can be found.  We start with some preliminary results.
 
 \begin{rem} \label{needit}
 	\normalfont
 	Let $M$ and $N$ be extensions of Tate motives such that the lowest bidegree of the slices of $M$ is strictly greater than the greatest bidegree of the slices of $N$ (with respect to the order in Definition \ref{ord}), then there are no non-trivial morphisms from $M$ to $N$, since the motivic cohomology groups of simplicial schemes are trivial in negative motivic weights, or motivic weight 0 and negative topological degrees. 
 \end{rem}
 
 \begin{lem}\label{eta1}
 	For any $i \leq i'$ and $j \leq j'$ satisfying $i \leq j$ and $i' \leq j'$, there exists a unique 
 	$$\eta: N^{i'+1}/N^{j'+1} \rightarrow N^{i+1}/N^{j+1}$$
 	such that the diagram
 	$$
 	\xymatrix{
 		N^{i'+1} \ar@{->}[r]\ar@{->}[d] & 	N^{i'+1}/N^{j'+1}  \ar@{->}[d]^{\eta}\\
 		N^{i+1} \ar@{->}[r] &  N^{i+1}/N^{j+1} 
 	}
 	$$ 
 	is commutative.
 \end{lem}
 \begin{proof}
 	Consider the diagram
 	$$
 	\xymatrix{
 		N^{j'+1}  \ar@{->}[r]\ar@{->}[d] & 	N^{i'+1}\ar@{->}[r]\ar@{->}[d] & N^{i'+1}/N^{j'+1} \ar@{->}[r]\ar@{->}[d]^{\eta} &N^{j'+1}[1] \ar@{->}[d]\\
 		N^{j+1}  \ar@{->}[r] & N^{i+1} \ar@{->}[r] & N^{i+1}/N^{j+1} \ar@{->}[r] & N^{j+1}[1]
 	}
 	$$ 
 	obtained by completing the leftmost commutative square to a morphism of distinguished triangles. The map $\eta$ is uniquely determined since there are no non-trivial morphisms from $N^{j'+1}[1]$ to $N^{i+1}/N^{j+1}$ by Remark \ref{needit}.
 \end{proof}
 
 \begin{lem}\label{cefunct}
 	For any $i \leq i' \leq i''$ and $j \leq j' \leq j''$ satisfying $i \leq j$, $i' \leq j'$ and $i'' \leq j''$, the composite
 	$$N^{i''+1}/N^{j''+1} \xrightarrow{\eta} N^{i'+1}/N^{j'+1} \xrightarrow{\eta} N^{i+1}/N^{j+1}$$
 	equals $\eta: N^{i''+1}/N^{j''+1} \rightarrow N^{i+1}/N^{j+1}$. 
 \end{lem}
 \begin{proof}
 	It follows easily from Lemma \ref{eta1}.
 \end{proof}
 
 For any $i \leq j \leq k$, denote by $\delta$ the composite $N^{i+1}/N^{j+1} \rightarrow N^{j+1}[1] \rightarrow N^{j+1}/N^{k+1}[1]$.
 
 \begin{lem}\label{cenat}
 	For any $i \leq i'$, $j \leq j'$ and $k \leq k'$ satisfying $i \leq j \leq k$ and $i' \leq j' \leq k''$, the diagram 
 	$$
 	\xymatrix{
 		N^{i'+1}/N^{j'+1} \ar@{->}[r]^{\delta}\ar@{->}[d]_{\eta} & 	N^{j'+1}/N^{k'+1}[1]  \ar@{->}[d]^{\eta}\\
 		N^{i+1}/N^{j+1} \ar@{->}[r]_{\delta} &  N^{j+1}/N^{k+1}[1] 
 	}
 	$$ 
 	is commutative.
 \end{lem}
 \begin{proof}
 	By Lemma \ref{eta1} we have two commutative squares
 	$$
 	\xymatrix{
 		N^{i'+1}/N^{j'+1} \ar@{->}[r]\ar@{->}[d]_{\eta} & 	N^{j'+1}[1] \ar@{->}[r]\ar@{->}[d] & 	N^{j'+1}/N^{k'+1}[1]  \ar@{->}[d]^{\eta}\\
 		N^{i+1}/N^{j+1} \ar@{->}[r]& 	N^{j+1}[1] \ar@{->}[r] &  N^{j+1}/N^{k+1}[1] 
 	}
 	$$ 
 	where the horizontal composites are both $\delta$ by definition.
 \end{proof}
 
 \begin{lem}\label{ceexact}
 	For any $i$, $j$ and $k$ satisfying $i \leq j \leq k$, there is a distinguished triangle
 	$$N^{j+1}/N^{k+1} \xrightarrow{\eta} N^{i+1}/N^{k+1} \xrightarrow{\eta}N^{i+1}/N^{j+1}\xrightarrow{\delta}N^{j+1}/N^{k+1}[1].$$
 \end{lem}
 \begin{proof}
 	Consider the diagram
 	$$
 	\xymatrix{
 		N^{j+1}  \ar@{->}[r]\ar@{->}[d] & 	N^{i+1}\ar@{->}[r]\ar@{->}[d] & N^{i+1}/N^{j+1} \ar@{->}[r]\ar@{=}[d] &N^{j+1}[1] \ar@{->}[d]\\
 		N^{j+1}/N^{k+1}  \ar@{->}[r]_{\eta} & N^{i+1}/N^{k+1} \ar@{->}[r] & N^{i+1}/N^{j+1} \ar@{->}[r] & N^{j+1}/N^{k+1}[1]
 	}
 	$$ 
 	obtained by completing the leftmost commutative square to a morphism of distinguished triangles. Then, by Lemma \ref{eta1} the middle bottom horizontal map is also $\eta$, while the right bottom horizontal map is $\delta$ by definition.
 \end{proof}
 
 \begin{prop}
 	The motivic cohomology groups
 	$$H^{**}(s,t)=H^{**}(N^{-t+1}/N^{-s+1})$$
 	for all $s \leq t$, together with the homomorphisms $\eta^*$ and $\delta^*$ respectively induced by $\eta$ and $\delta$, constitute a cohomological Cartan-Eilenberg system \cite[Definition 6.1.1]{rognes}.
 \end{prop}
 \begin{proof}
 	The result is a direct consequence of Lemmas \ref{cefunct}, \ref{cenat} and \ref{ceexact}.
 \end{proof}
 
 Our aim now is to produce a pairing of Cartan-Eilenberg systems. In order to do so, we need some technical lemmas. 
 
 \begin{lem}\label{fitis}
 	For any $i$, $j$ and $k$ satisfying $i+j \leq k$, there exists a unique 
 	$$f:(N \otimes N)/(N \otimes N)^{k+1} \rightarrow N/N^{i+1} \otimes N/N^{j+1}$$
 	such that the diagram
 	$$
 	\xymatrix{
 		N \otimes N \ar@{->}[r] \ar@{->}[d] & (N \otimes N)/(N \otimes N)^{k+1} \ar@{->}[ld]^{f}\\
 		N/N^{i+1} \otimes N/N^{j+1} & 
 	}
 	$$
 	is commutative.
 \end{lem}
 \begin{proof}
 	The map $f$ exists and is unique since there are no non-trivial morphisms from either $(N \otimes N)^{k+1}$ or $(N \otimes N)^{k+1}[1]$ to $N/N^{i+1} \otimes N/N^{j+1}$ by Remark \ref{needit}.
 \end{proof}
 
 From now on, we will denote by $\overline{\eta}$ and $\overline{\delta}$ the maps attached to the Postnikov system for $N \otimes N$.
 
 \begin{lem}\label{a}
 	For any $i\leq i'$, $j\leq j'$ and $k \leq k'$ satisfying $i+j \leq k$ and $i'+j' \leq k'$, the diagram
 	$$
 	\xymatrix{
 		(N \otimes N)/(N \otimes N)^{k'+1} \ar@{->}[r]^{\overline{\eta}} \ar@{->}[d]_{f} & (N \otimes N)/(N \otimes N)^{k+1} \ar@{->}[d]^{f}\\
 		N/N^{i'+1} \otimes N/N^{j'+1} \ar@{->}[r]_{\eta \otimes \eta} & N/N^{i+1} \otimes N/N^{j+1}
 	}
 	$$
 	is commutative and the composites equal $f$.
 \end{lem}
 \begin{proof}
 	It follows easily from Lemma \ref{fitis}.
 \end{proof}
 
 \begin{lem}\label{b}
 	For any $i\leq i'$, $j\leq j'$ and $k \leq k'$ satisfying $i+j' \leq k$, $i'+j \leq k$ and $i'+j' \leq k'$, there exists a unique 
 	$$F:(N \otimes N)^{k+1}/(N \otimes N)^{k'+1} \rightarrow N^{i+1}/N^{i'+1} \otimes N^{j+1}/N^{j'+1} $$ 
 	such that the diagram 
 	$$
 	\xymatrix{
 		(N \otimes N)^{k+1}/(N \otimes N)^{k'+1} \ar@{->}[r]^{\overline{\eta}} \ar@{->}[d]_{F} & (N \otimes N)/(N \otimes N)^{k'+1} \ar@{->}[d]^{f}\\
 		N^{i+1}/N^{i'+1} \otimes N^{j+1}/N^{j'+1} \ar@{->}[r]_{\eta \otimes \eta} & N/N^{i'+1} \otimes N/N^{j'+1}
 	}
 	$$
 	is commutative.
 \end{lem}
 \begin{proof}
 	Let $P$ be the homotopy pushout
 	$$
 	\xymatrix{
 		N^{i+1}/N^{i'+1} \otimes N^{j+1}/N^{j'+1} \ar@{->}[r]^{\eta \otimes id} \ar@{->}[d]_{id \otimes \eta} &  	N/N^{i'+1} \otimes N^{j+1}/N^{j'+1}\ar@{->}[d]^{\psi}\\
 		N^{i+1}/N^{i'+1} \otimes N/N^{j'+1} \ar@{->}[r]_{\phi} & P}
 	$$
 	and consider the diagrams
 	$$
 	\xymatrix{
 		(N \otimes N)/(N \otimes N)^{k'+1} \ar@{->}[r]^{\overline{\eta}}\ar@{->}[d]_{f} & 	(N \otimes N)/(N \otimes N)^{k+1}\ar@{->}[r]^-{\overline{\delta}}\ar@{->}[d]^{f} & (N \otimes N)^{k+1}/(N \otimes N)^{k'+1}[1] \ar@{->}[r]^-{\overline{\eta}}\ar@{->}[d]^{g} &\\
 		N/N^{i'+1} \otimes N/N^{j'+1} \ar@{->}[r]_{\eta \otimes \eta} & N/N^{i+1} \otimes N/N^{j+1} \ar@{->}[r] & P[1] \ar@{->}[r] & ,
 	}
 	$$ 
 	$$
 	\xymatrix{
 		(N \otimes N)/(N \otimes N)^{k'+1} \ar@{->}[r]^{\overline{\eta}}\ar@{->}[d]_{f} & 	(N \otimes N)/(N \otimes N)^{k+1}\ar@{->}[r]^-{\overline{\delta}}\ar@{->}[d]^{f} & (N \otimes N)^{k+1}/(N \otimes N)^{k'+1}[1] \ar@{->}[r]^-{\overline{\eta}}\ar@{->}[d]^{g'} &\\
 		N/N^{i'+1} \otimes N/N^{j'+1} \ar@{->}[r]_{\eta \otimes id} & N/N^{i+1} \otimes N/N^{j'+1} \ar@{->}[r]_-{\delta \otimes id} & N^{i+1}/N^{i'+1} \otimes N/N^{j'+1}[1] \ar@{->}[r]_-{\eta \otimes id} & 
 	}
 	$$ 
 	and
 	$$
 	\xymatrix{
 		(N \otimes N)/(N \otimes N)^{k'+1} \ar@{->}[r]^{\overline{\eta}}\ar@{->}[d]_{f} & 	(N \otimes N)/(N \otimes N)^{k+1}\ar@{->}[r]^-{\overline{\delta}}\ar@{->}[d]^{f} & (N \otimes N)^{k+1}/(N \otimes N)^{k'+1}[1] \ar@{->}[r]^-{\overline{\eta}}\ar@{->}[d]^{g''} &\\
 		N/N^{i'+1} \otimes N/N^{j'+1} \ar@{->}[r]_{id \otimes \eta} & N/N^{i'+1} \otimes N/N^{j+1} \ar@{->}[r]_-{id \otimes \delta} & N/N^{i'+1} \otimes N^{j+1}/N^{j'+1}[1] \ar@{->}[r]_-{id \otimes \eta} & 
 	}
 	$$ 
 	obtained by completing the leftmost commutative squares to morphisms of distinguished triangles. Note that the maps $g$, $g'$ and $g''$ are uniquely determined since there are no non-trivial morphisms from $(N \otimes N)^{k+1}/(N \otimes N)^{k'+1}[1]$ to either $N/N^{i+1} \otimes N/N^{j+1}$, $N/N^{i+1} \otimes N/N^{j'+1}$ or $N/N^{i'+1} \otimes N/N^{j+1}$ by Remark \ref{needit}. 
 	
 	Hence, the composite
 	$$(N \otimes N)^{k+1}/(N \otimes N)^{k'+1} \xrightarrow{ ({g' \atop g''})}  (N^{i+1}/N^{i'+1} \otimes N/N^{j'+1}) \oplus (N/N^{i'+1} \otimes N^{j+1}/N^{j'+1}) \xrightarrow{(\phi, -\psi)} P$$
 	is trivial since both $\phi g'$ and $\psi g''$ equal $g$. It follows that there exists a map
 	$$F:(N \otimes N)^{k+1}/(N \otimes N)^{k'+1} \rightarrow N^{i+1}/N^{i'+1} \otimes N^{j+1}/N^{j'+1} $$ 
 	such that the composite $({id \otimes \eta \atop \eta \otimes id}) \circ F$ equals $({g' \atop g''})$.
 	
 	Therefore, we have that
 	$$(\eta \otimes \eta) \circ F= (\eta \otimes id) \circ (id \otimes \eta) \circ F=(\eta \otimes id) \circ g'= f \circ \overline{\eta}.$$
 	
 	Let $P'$ be the homotopy pullback
 	$$
 	\xymatrix{
 		P'\ar@{->}[r] \ar@{->}[d] &  	N/N^{i'+1} \otimes N/N^{j+1}\ar@{->}[d]^{\eta \otimes id}\\
 		N/N^{i+1} \otimes N/N^{j'+1} \ar@{->}[r]_{id \otimes \eta} & 	N/N^{i+1} \otimes N/N^{j+1}}
 	$$
 	and consider the distinguished triangle
 	$$N^{i+1}/N^{i'+1} \otimes N^{j+1}/N^{j'+1} \xrightarrow{\eta \otimes \eta} N/N^{i'+1} \otimes N/N^{j'+1} \rightarrow P' \rightarrow N^{i+1}/N^{i'+1} \otimes N^{j+1}/N^{j'+1}[1].$$
 	Since there are no non-trivial morphisms from $(N \otimes N)^{k+1}/(N \otimes N)^{k'+1}$ to $P'[-1]$ by Remark \ref{needit}, we conclude that $F$ is uniquely determined.
 \end{proof}
 
 \begin{prop}\label{spp1}
 	For any $i\leq i'$, $j\leq j'$, $k \leq k'$, $l \leq l'$, $m \leq m'$ and $n \leq n'$ satisfying $i+j' \leq k$, $i'+j \leq k$, $i'+j' \leq k'$, $l+m' \leq n$, $l'+m \leq n$, $l'+m' \leq n'$, $i\leq l$, $j\leq m$, $k \leq n$, $i'\leq l'$, $j'\leq m'$ and $k' \leq n'$, the diagram
 	$$
 	\xymatrix{
 		(N \otimes N)^{n+1}/(N \otimes N)^{n'+1} \ar@{->}[r]^{\overline{\eta}} \ar@{->}[d]_{F} & (N \otimes N)^{k+1}/(N \otimes N)^{k'+1} \ar@{->}[d]^{F}\\
 		N^{l+1}/N^{l'+1} \otimes N^{m+1}/N^{m'+1} \ar@{->}[r]_{\eta \otimes \eta} & N^{i+1}/N^{i'+1} \otimes N^{j+1}/N^{j'+1}
 	}
 	$$
 	is commutative and the composites equal $F$.
 \end{prop}
 \begin{proof}
 	It follows easily from Lemma \ref{b}.
 \end{proof}

 \begin{prop}\label{spp2}
 	For any $i$, $j$ and $r \geq 1$, in the diagram
 	$$
 	\xymatrixcolsep{0.001pc}
 	\xymatrix{
 		(N \otimes N)^{i+j-r}/(N \otimes N)^{i+j-r+1} \ar@{->}[rr]^{F} \ar@{->}[dd]_{F} \ar@{->}[dr]^{\overline{\delta}}& & N^{i}/N^{i+1} \otimes N^{j-r}/N^{j-r+1} \ar@{->}[dd]_{\eta \otimes \delta}\\
 		& (N\otimes N)^{i+j-r+1}/(N \otimes N)^{i+j+1}[1]\ar@{->}[dr]^{F}&\\
 		N^{i-r}/N^{i-r+1} \otimes N^{j}/N^{j+1} \ar@{->}[rr]_{\delta \otimes \eta} &&N^{i-r+1}/N^{i+1} \otimes N^{j-r+1}/N^{j+1}[1]
 	}
 	$$
 	the diagonal composite equals the sum of the
 	two outer composites.
 \end{prop}
 \begin{proof}
 	Let $P'$ be the homotopy pullback
 	$$
 	\xymatrix{
 		P'\ar@{->}[r]^{\psi'} \ar@{->}[d]_{\phi'} &  	N^{i-r}/N^{i-r+1} \otimes N^{j-r}/N^{j+1}\ar@{->}[d]^{id \otimes \eta}\\
 		N^{i-r}/N^{i+1} \otimes N^{j-r}/N^{j-r+1} \ar@{->}[r]_{\eta \otimes id} & 	N^{i-r}/N^{i-r+1} \otimes N^{j-r}/N^{j-r+1}}
 	$$
 	and consider the distinguished triangle
 	$$N^{i-r+1}/N^{i+1} \otimes N^{j-r+1}/N^{j+1} \xrightarrow{\eta \otimes \eta} N^{i-r}/N^{i+1} \otimes N^{j-r}/N^{j+1} \rightarrow P' \xrightarrow{\partial} N^{i-r+1}/N^{i+1} \otimes N^{j-r+1}/N^{j+1}[1].$$
 	Then, the commutative square
 	$$
 	\xymatrix{
 		(N \otimes N)^{i+j-r+1}/(N \otimes N)^{i+j+1} \ar@{->}[r]^{\overline{\eta}}\ar@{->}[d]_{F} & 	(N \otimes N)^{i+j-r}/(N \otimes N)^{i+j+1}\ar@{->}[d]^{F} \\
 		N^{i-r+1}/N^{i+1} \otimes N^{j-r+1}/N^{j+1} \ar@{->}[r]_{\eta \otimes \eta} & N^{i-r}/N^{i+1} \otimes N^{j-r}/N^{j+1}  
 	}
 	$$ 
 	induces a unique $G:(N \otimes N)^{i+j-r}/(N \otimes N)^{i+j-r+1} \rightarrow P'$ on the cones, since there are no non-trivial morphisms from $(N \otimes N)^{i+j-r+1}/(N \otimes N)^{i+j+1}[1]$ to $P'$ by Remark \ref{needit}. 
 	
 	Note that the map
 	$$(N^{i}/N^{i+1}\otimes N^{j-r}/ N^{j-r+1})\oplus(N^{i-r}/N^{i-r+1}\otimes N^{j}/ N^{j+1}) \xrightarrow{({\eta \otimes id \atop id \otimes \eta})} N^{i-r}/N^{i-r+1}\otimes N^{j-r}/N^{j-r+1}$$
 	factoring through $(N^{i-r}/N^{i+1}\otimes N^{j-r}/ N^{j-r+1})\oplus(N^{i-r}/N^{i-r+1}\otimes N^{j-r}/ N^{j+1})$ is trivial by Remark \ref{needit}. Hence, there exists a unique 
 	$$(\alpha,\beta): (N^{i}/N^{i+1}\otimes N^{j-r}/ N^{j-r+1})\oplus(N^{i-r}/N^{i-r+1}\otimes N^{j}/ N^{j+1}) \rightarrow P'$$
 	such that $({\phi' \atop \psi'}) \circ (\alpha,\beta) = ({\eta \otimes id \atop 0}{0 \atop id \otimes \eta})$. For a similar reason, there exists a unique $G':(N \otimes N)^{i+j-r}/(N \otimes N)^{i+j-r+1} \rightarrow P'$ such that $({\phi' \atop \psi'}) \circ G' = ({F \atop F})$. Since both $({\phi' \atop \psi'}) \circ G = ({F \atop F})$ and $({\phi' \atop \psi'}) \circ (\alpha,\beta) \circ ({F \atop F})= ({F \atop F})$, it follows that $G = (\alpha,\beta) \circ ({F \atop F})$.
 	
 	Therefore, we deduce that
 	$$F \circ \overline{\delta}=\partial \circ G= \partial \circ (\alpha,\beta) \circ \left({F \atop F}\right).$$
 	At this point we only need to identify $\partial \circ (\alpha,\beta)$ with $(\eta \otimes \delta,\delta \otimes \eta)$. 
 	
 	Note that the commutative square
 	$$
 	\xymatrix{
 		N^i/N^{i+1} \otimes N^{j-r+1}/N^{j+1}\ar@{->}[r]^{id \otimes \eta }\ar@{->}[d]_{\eta \otimes id} & N^i/N^{i+1} \otimes N^{j-r}/N^{j+1}\ar@{->}[d]^{\eta \otimes id} \\
 		N^{i-r+1}/N^{i+1} \otimes N^{j-r+1}/N^{j+1} \ar@{->}[r]_{\eta \otimes \eta} & N^{i-r}/N^{i+1} \otimes N^{j-r}/N^{j+1}
 	}
 	$$ 
 	induces a unique map $\alpha': N^i/N^{i+1} \otimes N^{j-r}/N^{j-r+1} \rightarrow P'$ on the cones since there are no non-trivial morphisms from $N^i/N^{i+1} \otimes N^{j-r+1}/N^{j+1}[1]$ to $P'$ by Remark \ref{needit}. For the same reason we have a unique $\beta': N^{i-r}/N^{i-r+1} \otimes N^{j}/N^{j+1} \rightarrow P'$, and $\partial \circ (\alpha',\beta')=(\eta \otimes \delta,\delta \otimes \eta)$. On the other hand $({\phi' \atop \psi'}) \circ (\alpha',\beta') = ({\eta \otimes id \atop 0}{0 \atop id \otimes \eta})$, hence $(\alpha',\beta')=(\alpha,\beta)$ that completes the proof.
 \end{proof}
 
 \begin{thm}\label{mult}
 	Let $\pi:X_{\bullet,\bullet} \rightarrow Y_{\bullet,\bullet}$ be a smooth coherent morphism satisfying all conditions of Proposition \ref{Serre} with a collinearly weighted $M(A_{\bullet})$. Then, the spectral sequence, induced by the Postnikov system parametrized with the corresponding maximal collinear set of weights, is multiplicative, i.e. there are products $\cdot_r:E^i_r \otimes E^j_r \rightarrow E^{i+j}_r$ such that
 	$$d_r^{i+j}(a\cdot_r b) = d_r^i(a)\cdot_r b + (-1)^{|a|}a \cdot_r d_r^j(b)$$
 	where $a$ and $b$ are classes in $E_r^i$ and $E_r^j$ respectively, and $|a|$ is the topological degree of $a$. 
 \end{thm}
 \begin{proof}
 	For all $r \geq 1$, denote by $\Phi$ the composite
 	$$N^{i+j-r+1}/N^{i+j+1} \rightarrow (N \otimes N)^{i+j-r+1}/(N \otimes N)^{i+j+1} \xrightarrow{F} N^{i-r+1}/N^{i+1} \otimes N^{j-r+1}/N^{j+1}.$$ 
 	By Propositions $\ref{spp1}$ and \ref{spp2}, the induced morphisms in motivic cohomology $\Phi^*$ provide us with a pairing of Cartan-Eilenberg systems $\mu: ((H^{**},\eta^*,\delta^*),(H^{**},\eta^*,\delta^*)) \rightarrow (H^{**},\eta^*,\delta^*)$ \cite[Definition 6.2.1]{rognes}. Then, the result follows from \cite[Theorem 6.2.3]{rognes}.
 \end{proof}
 
 \section{The case of $BPGL_n$}
 
 In this section we want to apply the spectral sequence of Theorem \ref{Serre2} to approach the computation of the motivic cohomology of the Nisnevich classifying space of $PGL_n$. From now on we assume that the base field $k$ has characteristic not dividing $n$.
 
 First, we recall a few definitions. For a simplicial algebraic group $G_{\bullet}$ denote by $EG_{\bullet}$ the weakly contractible bisimplicial scheme defined by 
 $$(EG_{\bullet})_{i,h} = G_{h}^{i+1}$$
 with face and degeneracy maps given respectively by partial projections and partial diagonals along $i$ and by the appropriate power of face and degeneracy maps of $G_{\bullet}$ along $h$. Note that $EG_{\bullet}$ has a right free $G_{\bullet}$-action. Denote by $BG_{\bullet}$ the bisimplicial scheme obtained as a quotient of $EG_{\bullet}$ by this action. On each simplicial component $BG_{\bullet}$ looks like
 $$(BG_{\bullet})_{i,h} = G_{h}^{i}.$$
 By abuse of notation, we denote by $EG_{\bullet}$ and $BG_{\bullet}$ also the diagonals of the respective bisimplicial schemes.
 
 \begin{dfn}
 	The simplicial scheme $BG_{\bullet}$ is called the Nisnevich classifying space of $G_{\bullet}$ (\cite[Example 1.11]{morel.voevodsky}).
 \end{dfn}
 
 \begin{rem}
 	\normalfont
 	Note that the map of bisimplicial schemes $EG_{\bullet} \rightarrow BG_{\bullet}$ is smooth coherent and over the 0th simplicial component is the projection $G_{\bullet} \rightarrow Spec(k)$. This is the reason why we need to work with bisimplicial models. Indeed, these maps provide a rich source of examples where it is possible to apply the spectral sequence in Theorem \ref{Serre2}. In fact, if $G$ is a commutative algebraic group, then for any $n \geq 1$ there are simplicial commutative algebraic groups $B^nG$ each of which is the (diagonal of the) classifying space of the previous one. So, we get coherent morphisms of bisimplicial schemes $EB^nG \rightarrow B^{n+1}G$ with fiber $B^nG$. In particular, if $G=\Gm$, then $B^nG$ is a motivic Eilenberg-MacLane space $K(\Z,n+1,1)$, and we know from \cite{voevodsky.em} that their motives are cellular, i.e. direct sum of Tate motives. Hence, we get Serre spectral sequences for the motivic cohomology of Eilenberg-MacLane spaces of this type.
 \end{rem}
 
 From the short exact sequence of algebraic groups
 \begin{align}\label{cenext}
 	1 \rightarrow \Gm \rightarrow GL_n \rightarrow PGL_n \rightarrow 1
 \end{align}
 one gets the following fiber sequence
 $$B\Gm \rightarrow BGL_n \rightarrow BPGL_n.$$
 For our purposes, consider $EB\Gm \times BGL_n$ and $(EB\Gm \times BGL_n)/B\Gm$ as bisimplicial models for $BGL_n$ and $BPGL_n$ respectively (see \cite[Example 9.11]{jardine}). The smooth coherent morphisms of bisimplicial schemes
 $$EB\Gm \times BGL_n \rightarrow  (EB\Gm \times BGL_n)/B\Gm$$
 and 
 $$EB\Gm \rightarrow BB\Gm$$
 are both trivial projections over the 0th simplicial component with fiber $B\Gm$. Hence, we get a cartesian square of bisimplicial schemes
 \begin{align}\label{square}
 	\xymatrix{
 		EB\Gm\times BGL_n \ar@{->}[r] \ar@{->}[d] & (EB\Gm \times BGL_n)/B\Gm \ar@{->}[d]\\
 		EB\Gm \ar@{->}[r] & BB\Gm.
 	}
 \end{align}
 
 Recall from \cite[Proposition 3.7]{morel.voevodsky} that $B\Gm$ is $A^1$-homotopy equivalent to $P^{\infty}$ whose motive is cellular. Indeed, we have that $M(P^{\infty}) \cong \bigoplus_{j=0}^{\infty} T(j)[2j]$. Therefore, we can apply the spectral sequence constructed in Theorem \ref{Serre2} to the $BPGL_n$ case, which leads to the following result.
 
 \begin{thm}\label{bpg}
 	There exists a strongly convergent spectral sequence
 	$$E_1^{p,q,s}=  H^{p-2s,q-s}(BPGL_n) \Longrightarrow H^{p,q}(BGL_n)$$
 	with differentials $d_r^{p,q,s}:E_r^{p,q,s} \rightarrow E_r^{p+1,q,s-r}$. Moreover, the differential
 	$$d_1^{p,q,s}:H^{p-2s,q-s}(BPGL_n) \rightarrow H^{p-2s+3,q-s+1}(BPGL_n)$$
 	is the multiplication by $s\cdot d_1^{2,1,1}(1)$.
 \end{thm}
 \begin{proof}
 	Applying Theorem \ref{Serre2} to the map $EB\Gm \times BGL_n \rightarrow  (EB\Gm \times BGL_n)/B\Gm$ gives the strongly convergent spectral sequence. The description of the first differential follows easily by induction on $s$. In fact, suppose $d_1^{2s-2,s-1,s-1}(1)=(s-1) \cdot d_1^{2,1,1}(1)$, then from Theorem \ref{mult} we deduce that
 	$$d_1^{2s,s,s}(1)=d_1^{2s-2,s-1,s-1}(1)+d_1^{2,1,1}(1)=s \cdot d_1^{2,1,1}(1)$$
 	which concludes the proof.
 \end{proof}
 
 Since the motivic cohomology of $BGL_n$ is known, i.e. $H^{**}(BGL_n) \cong H^{**}(k)[c_1,\dots,c_n]$, we can ``reverse-engineer" the previous spectral sequence in order to obtain information about the motivic cohomology of $BPGL_n$.
 
 Before proceeding, recall that the Chern class $c_i$ is in bidegree $(i)[2i]$ for any $i$, so $H^{p,q}(BGL_n) \cong 0$ for $p > 2q$.
 
 \begin{cor}\label{triv}
 	For all $p \geq 3q+1$ we have that $H^{p.q}(BPGL_n) \cong 0$.
 \end{cor}
 \begin{proof}
 	We proceed by induction on $q$. For $q=0$, it follows from an easy inspection of the spectral sequence that $H^{p,0}(BPGL_n) \cong H^{p,0}(BGL_n) \cong 0$ for all $p \geq 1$, which provides the induction basis.
 	
 	Now, suppose that the statement holds for all motivic weights less than $q$. The $E_1$-page of the spectral sequence is
 	$$E_1^{p,q,s}=  H^{p-2s,q-s}(BPGL_n).$$
 	Consider $p \geq 3q+1$, then $p-2s \geq 3q+1-2s \geq 3(q-s) +1$, Hence, by induction hypothesis $E_1^{p,q,s} \cong 0$ for all $s \geq 1$. It follows that the only piece of the spectral sequence that contributes for $p \geq 3q+1$ to $H^{p,q}(BGL_n) \cong 0$ comes from $E_1^{p,q,0}=  H^{p,q}(BPGL_n)$. But the differentials $d_r: E_r^{p-1,q,r} \rightarrow E_r^{p,q,0}$ are all trivial since $E_1^{p-1,q,r} \cong H^{p-1-2r,q-r}(BPGL_n) \cong 0$ as $p-1-2r \geq 3q-2r \geq 3(q-r)+1$.
 	
 	Therefore, 
 	$$H^{p,q}(BGL_n) \cong E_{\infty}^{p,q,0} \cong H^{p,q}(BPGL_n)$$
 	for $p \geq 3q+1$ that concludes the proof.
 \end{proof}
 
 Recall from \cite[Example 9.11]{jardine} that $BB\Gm$ is the Eilenberg-MacLane space $K(\Gm,2)$, so by adjunction there is a canonical element $\chi$ in $H^{3,1}(BB\Gm) \cong H_{Nis}^2(BB\Gm,\Gm) \cong [BB\Gm,BB\Gm]$ corresponding to the identity $BB\Gm \rightarrow BB\Gm$ (here by $[-,-]$ we mean hom-sets in ${\mathcal H}_s(k)$).
 
 \begin{lem}\label{bb}
 	We have that
 	$$H^{3,1}(BB\Gm) \cong \Z$$
 	generated by $\chi$.
 \end{lem}
 \begin{proof}
 	We can apply Theorem \ref{Serre2} to the coherent morphism $EB\Gm \rightarrow BB\Gm$ with fiber $B\Gm$. Note that for $q=0$ the differentials are all trivial and we get
 	$$H^{p,0}(EB\Gm) \cong E_{\infty}^{p,0,0} \cong H^{p,0}(BB\Gm)$$
 	from which it follows that $H^{0,0}(BB\Gm) \cong \Z$ and $H^{p,0}(BB\Gm) \cong 0$ for $p \neq 0$.
 	
 	In order to compute $H^{3,1}(BB\Gm)$, since $E_1^{3,1,1} \cong H^{1,0}(BB\Gm) \cong 0$, the part of the $E_1$-page we need consists only of the groups $E_1^{3,1,0} \cong H^{3,1}(BB\Gm)$ and $E_1^{2,1,1} \cong H^{0,0}(BB\Gm) \cong \Z$ linked by the differential $d_1^{2,1,1}:H^{0,0}(BB\Gm) \cong \Z \rightarrow H^{3,1}(BB\Gm)$. Hence, we obtain 
 	$$0\cong H^{3,1}(EB\Gm) \cong E_{\infty}^{3,1,0} \cong H^{3,1}(BB\Gm)/\Ima(d_1^{2,1,1}),$$
 	$$E_{\infty}^{2,1,0} \cong H^{2,1}(BB\Gm)$$
 	and
 	$$E_{\infty}^{2,1,1} \cong  \ker(d_1^{2,1,1}).$$
 	Therefore, from the short exact sequence
 	$$0 \rightarrow E_{\infty}^{2,1,0} \rightarrow H^{2,1}(EB\Gm) \cong 0 \rightarrow E_{\infty}^{2,1,1} \rightarrow 0$$
 	one gets that $d_1^{2,1,1}$ is an isomorphism, which completes the proof.
 \end{proof}
 
 The right vertical map in (\ref{square}) induces in ${\mathcal H}_s(k)$ a class of $[BPGL_n,BB\Gm] \cong H^{3,1}(BPGL_n)$ that classifies the central extension (\ref{cenext}) (see \cite[Theorem 1.2]{rolle}). Denote by $x$ this canonical element. Note that $x$ is nothing but the image of $\chi$ under the induced homomorphism $H^{3,1}(BB\Gm) \rightarrow H^{3,1}(BPGL_n)$.
 
 \begin{thm}\label{comp}
 	In motivic weights 0, 1 and 2 the following isomorphisms hold
 	\begin{align*}
 		H^{p,0}(BPGL_n) &\cong 
 		\begin{cases}
 			\Z & p=0\\
 			0 & otherwise
 		\end{cases}\\
 		H^{p,1}(BPGL_n) &\cong 
 		\begin{cases}
 			k^* & p=1\\
 			\Z/n \cdot x & p=3\\
 			0 & otherwise
 		\end{cases}\\
 		H^{p,2}(BPGL_n) &\cong 
 		\begin{cases}
 			H^{p,2}(k) & p\leq2\\
 			\mu_n(k) & p=3\\
 			k^*/n \cdot x \oplus \Z & p=4\\
 			\Z/2 \cdot x^2 & p=6 \: and \: n \: even\\
 			0 & otherwise.
 		\end{cases}
 	\end{align*}
 \end{thm}
 \begin{proof}
 	The result follows from the spectral sequence in Theorem \ref{bpg}. 
 	
 	Let us start from the case $q=0$. Then, the only possibly non-trivial groups in the $E_1$-page are $E_1^{p,0,0} \cong H^{p,0}(BPGL_n)$ for $p \geq 0$. In this case differentials are all trivial and we get
 	$$H^{p,0}(BGL_n) \cong E_{\infty}^{p,0,0} \cong H^{p,0}(BPGL_n)$$
 	from which it follows the motivic weight $0$ case.
 	
 	For the case $q=1$, the non-trivial part of the $E_1$-page possibly consists of the groups $E_1^{p,1,0} \cong H^{p,1}(BPGL_n)$ for $p \geq 1$ and $E_1^{2,1,1} \cong H^{0,0}(BPGL_n) \cong \Z$. There is only one non-zero differential $d_1^{2,1,1}:H^{0,0}(BPGL_n) \cong \Z \rightarrow H^{3,1}(BPGL_n)$. Hence, we obtain 
 	$$H^{p,1}(BGL_n) \cong E_{\infty}^{p,1,0} \cong H^{p,1}(BPGL_n)$$
 	for $p \neq 2,3$,
 	$$0\cong H^{3,1}(BGL_n) \cong E_{\infty}^{3,1,0} \cong H^{3,1}(BPGL_n)/\Ima(d_1^{2,1,1}),$$
 	$$E_{\infty}^{2,1,0} \cong H^{2,1}(BPGL_n)$$
 	and
 	$$E_{\infty}^{2,1,1} \cong  \ker(d_1^{2,1,1}).$$
 	Therefore, from the short exact sequence
 	$$0 \rightarrow E_{\infty}^{2,1,0} \rightarrow H^{2,1}(BGL_n) \rightarrow E_{\infty}^{2,1,1} \rightarrow 0$$
 	one gets the exact sequence
 	$$0 \rightarrow H^{2,1}(BPGL_n) \rightarrow \Z \rightarrow \Z \xrightarrow{d_1^{2,1,1}} H^{3,1}(BPGL_n) \rightarrow 0.$$
 	At this point, we only need to understand the homomorphism in the middle $\Z \rightarrow \Z$. Note that the latter is just the homomorphism $H^{2,1}(BGL_n) \rightarrow H^{2,1}(N^1)$ induced by the Postnikov system generating the spectral sequence. Recall that $H^{2,1}(BGL_n)$ is generated by the first Chern class $c_1$ while $H^{2,1}(N^1) \cong H^{2,1}(B\Gm)$ is generated by the Chern class $c$. Since the map $B\Gm \rightarrow BGL_n$ factors through $(B\Gm)^n$ we have that the previous homomorphism maps $c_1$ to $nc$. It follows that $H^{2,1}(BPGL_n) \cong 0$ and $H^{3,1}(BPGL_n) \cong \Z/n$ is generated by $x=d_1^{2,1,1}(1)$, by Lemma \ref{bb} and the functoriality of the spectral sequence.
 	
 	For the case $q=2$, we have $E_1^{p,2,0} \cong H^{p,2}(BPGL_n)$, $E_1^{3,2,1} \cong H^{1,1}(BPGL_n) \cong k^*$, $E_1^{5,2,1} \cong H^{3,1}(BPGL_n) \cong \Z/n$ and $E_1^{4,2,2} \cong H^{0,0}(BPGL_n) \cong \Z$. The possibly non-trivial differentials on the $E_1$-page are $d_1^{4,2,2}$, $d_1^{3,2,1}$ and $d_1^{5,2,1}$. Note that, by Theorem \ref{bpg}, $d_1^{4,2,2}$ is the multiplication by $2x$ and $d_1^{5,2,1}$ is surjective since $H^{6,2}(BGL_n)$ is trivial.
 	
 	From the short exact sequence
 	$$0 \rightarrow E_{\infty}^{5,2,0}\rightarrow H^{5,2}(BGL_n) \rightarrow E_{\infty}^{5,2,1} \rightarrow 0$$
 	and since $H^{5,2}(BGL_n) \cong 0$ we get that both  $E_{\infty}^{5,2,0} \cong H^{5,2}(BPGL_n)/ \Ima(d_2^{4,2,2})$ and $E_{\infty}^{5,2,1} \cong \ker(d_1^{5,2,1})/ \Ima(d_1^{4,2,2})$ are trivial. In particular, $d_2^{4,2,2}$ is surjective and the complex
 	$$H^{0,0}(BPGL_n) \xrightarrow{\cdot 2x} H^{3,1}(BPGL_n) \xrightarrow{\cdot x} H^{6,2}(BPGL_n) \rightarrow 0$$
 	is exact. The first homomorphism of the latter complex is $\Z \xrightarrow{\cdot 2} \Z/n$. Hence, when $n$ is odd, it is surjective and $H^{6,2}(BPGL_n) \cong 0$, while, when $n$ is even, its image is $\Z/({\frac n 2})$ and $H^{6,2}(BPGL_n) \cong \Z/2$ generated by $x^2$. 
 	
 	From the short exact sequence
 	$$0 \rightarrow E_{\infty}^{3,2,0}\rightarrow H^{3,2}(BGL_n) \rightarrow E_{\infty}^{3,2,1} \rightarrow 0$$
 	and by recalling that $c_1$ is mapped to $nc$ via $H^{2,1}(BGL_n) \rightarrow H^{2,1}(N^1)$, we get the exact sequence
 	$$0 \rightarrow H^{3,2}(BPGL_n) \rightarrow k^* \xrightarrow{\cdot n} k^* \rightarrow k^*/n \rightarrow 0$$
 	where the homomorphism in the middle can be identified with $H^{3,2}(BGL_n) \rightarrow H^{3,2}(N^1)$. Hence, $H^{3,2}(BPGL_n) \cong \mu_n(k)$.
 	
 	Finally, from the short exact sequence
 	$$0 \rightarrow E_{\infty}^{4,2,0}\rightarrow H^{4,2}(BGL_n) \rightarrow E_{\infty}^{4,2,2} \rightarrow 0$$
 	we get the exact sequence
 	$$0 \rightarrow H^{4,2}(BPGL_n)/\Ima(d_1^{3,2,1}) \rightarrow \Z \oplus \Z \rightarrow E_{\infty}^{4,2,2} \rightarrow 0.$$
 	Note that $E_{\infty}^{4,2,2}$ is a subgroup of $H^{4,2}(N^2)\cong \Z$. The latter is generated by $c^2$ and the homomorphism $H^{4,2}(BGL_n) \rightarrow E_{\infty}^{4,2,2}$ maps $c_1^2$ to $n^2c^2$ and $c_2$ to ${\frac{n(n-1)}2} c^2$. At this point we want to prove that $d_2^{4,2,2}$ is trivial. To this end, it is enough to prove that the homomorphism $H^{4,2}(BGL_n) \rightarrow H^{4,2}(N^1)$ is surjective. First, note that since $H^{2,1}(BPGL_n) \cong 0$ then the homomorphism $H^{4,2}(N^1) \rightarrow H^{4,2}(N^2)$ induced by the Postnikov system is injective. Hence, we get an exact sequence
 	$$0 \rightarrow H^{4,2}(N^1) \cong \Z \rightarrow H^{4,2}(N^2) \cong \Z \xrightarrow{\cdot 2} H^{3,1}(BPGL_n) \cong \Z/n$$
 	from which it follows that $H^{4,2}(N^1) \cong \Z$ is generated by an element $z$ mapping to $nc^2$, if $n$ is odd, and to ${\frac n 2}c^2$, if $n$ is even, in $H^{4,2}(N^2)$. But $c_1^2-2c_2$ in $H^{4,2}(BGL_n)$ maps to $nc^2$ in $H^{4,2}(N^2)$ if $n$ is odd, while ${\frac n 2}c_1^2-(n+1)c_2$ maps to ${\frac n 2}c^2$ if $n$ is even. Therefore, $d_2^{4,2,2}$ is trivial and surjective, so $H^{5,2}(BPGL_n) \cong 0$. It immediately follows that $H^{4,2}(BPGL_n) \cong k^*/n \cdot x \oplus \Z$ where the generator of $\Z$ maps to $(n-1)c_1^2-2nc_2$ if $n$ is even, and to ${\frac {n-1} 2}c_1^2-nc_2$ if $n$ is odd. This concludes the proof. 
 \end{proof}
 
 The next result tells us that, as expected, the interesting part of $H^{**}(BPGL_n)$ is $n$-torsion.
 
 \begin{prop}
 	There are isomorphisms of $H^{**}(k,\Z[{\frac 1 n}])$-algebras
 	$$H^{**}(BPGL_n,\Z[{\tfrac 1 n}]) \cong H^{**}(BSL_n,\Z[{\tfrac 1 n}]) \cong  H^{**}(k,\Z[{\tfrac {1} {n}}])[c_2,\dots,c_n].$$
 \end{prop}
 \begin{proof}
 	The second isomorphism is well-known (already with $\Z$-coefficients), so we only need to show the first one. 
 	
 	Since the standard morphism $GL_n \rightarrow PGL_n$ is a $\Gm$-torsor we have a cartesian square
 	$$
 	\xymatrix{
 		GL_n \times \Gm \ar@{->}[r]^{\pi} \ar@{->}[d]_{\alpha} & GL_n \ar@{->}[d]\\
 		GL_n \ar@{->}[r] & PGL_n
 	}
 	$$
 	where $\pi$ is the projection and $\alpha$ is the $\Gm$-action. The latter induces in turn a cartesian square
 	$$
 	\xymatrix{
 		SL_n \times \Gm \ar@{->}[r]^{\pi} \ar@{->}[d]_{\tilde{\alpha}} & SL_n \ar@{->}[d]\\
 		GL_n \ar@{->}[r] & PGL_n
 	}
 	$$
 	where the morphism $SL_n \rightarrow PGL_n$ (factoring through $GL_n$) is the usual $\mu_n$-torsor.
 	
 	Note that $\tilde{\alpha}$ induces a homomorphism on the motivic cohomology of the respective classifying spaces $H^{**}(BGL_n) \rightarrow H^{**}(BSL_n) \otimes_{H^{**}(k)} H^{**}(B\Gm)$ that maps the total Chern class $c_{\bullet}(t)=\sum_{i=0}^n c_it^{n-i}$ to $c_{\bullet}(t+c)$. Hence, we get an isomorphism 
 	$$B\tilde{\alpha}^{*}:H^{**}(BGL_n,\Z[{\tfrac 1 n}]) \cong H^{**}(BSL_n,\Z[{\tfrac 1 n}]) \otimes_{H^{**}(k,\Z[{\frac 1 n}])}H^{**}(B\Gm,\Z[{\tfrac 1 n}]).$$
 	
 	Now we want to prove by induction on the motivic weight that the homomorphism 
 	$$H^{**}(BPGL_n,\Z[{\tfrac 1 n}]) \rightarrow H^{**}(BSL_n,\Z[{\tfrac 1 n}])$$ 
 	is an isomorphism. We use the functoriality of the Postnikov systems provided by Proposition \ref{Serre 3}. For $q=0$ and all $p$, our spectral sequence implies that 
 	$$H^{p,0}(BPGL_n,\Z[{\frac 1 n}]) \cong H^{p,0}(BGL_n,\Z[{\frac 1 n}]) \cong H^{p,0}(BSL_n,\Z[{\frac 1 n}])$$
 	which provides the induction basis. Suppose that $H^{p,q'}(BPGL_n,\Z[{\frac 1 n}]) \cong H^{p,q'}(BSL_n,\Z[{\frac 1 n}])$ for all $q' < q$ and all $p$. Since both $H^{p,q}(BPGL_n,\Z[{\frac 1 n}])$ and $H^{p,q}(BSL_n,\Z[{\frac 1 n}])$ can be reconstructed respectively from compatible extensions of $H^{p,q'}(BPGL_n,\Z[{\frac 1 n}])$ and $H^{p,q'}(BSL_n,\Z[{\frac 1 n}])$ for $q' < q$ and $H^{p,q}(BGL_n,\Z[{\frac 1 n}])$, five lemma implies that 
 	$$H^{p,q}(BPGL_n,\Z[{\tfrac 1 n}]) \rightarrow H^{p,q}(BSL_n,\Z[{\tfrac 1 n}])$$
 	is an isomorphism that is what we aimed to show.
 \end{proof}
 
 \section{The motive of a Severi-Brauer variety}
 
 The purpose of this section is to apply previous results to obtain a description of the motive of a Severi-Brauer variety.
 
 Let $A$ be a central simple algebra of degree $n$ and $\check C(\seb(A))$ be the \v{C}ech simplicial scheme of the respective Severi-Brauer variety $\seb(A)$, i.e. $\check C(\seb(A))_n=\seb(A)^{n+1}$ with face and degeneracy maps given respectively by partial projections and diagonals. Moreover, denote by $\X_A$ the motive of $\check C(\seb(A))$ in $\DM_{eff}^{-}(k)$ and by $X_A$ the $PGL_n$-torsor associated to $A$, i.e. $X_A = {\mathrm Iso}\{A \leftrightarrow M_n(k)\}$. 
 
 \begin{rem}\label{long}
 	\normalfont
 	Since $X_A$ is a form of $PGL_n$, then the scheme $X_A/P$ is a Severi-Brauer variety for $A$, i.e.
 	$$\seb(A) \cong X_A/P$$
 	where $P$ is the parabolic subgroup of $PGL_n$ that stabilizes the point $[0,\dots,0,1]$ in $P^{n-1}$. Similarly, let $\widetilde{P}$ be the parabolic subgroup of $GL_n$ that stabilizes $[0,\dots,0,1]$. Hence, we have a cartesian square 
 	$$
 	\xymatrix{
 		\widetilde{P} \ar@{->}[r] \ar@{->}[d] & P \ar@{->}[d]\\
 		GL_n \ar@{->}[r] & PGL_n
 	}
 	$$
 	where the top horizontal map is a split $\Gm$-torsor, since $\widetilde{P}$ and $P$ can be respectively identified with $A^{n-1}\rtimes (GL_{n-1} \times \Gm)$ and $A^{n-1}\rtimes GL_{n-1}$. It follows that the inclusion $P \hookrightarrow PGL_n$ factors through $GL_n$ so that we have a sequence of maps
 	$$GL_{n-1} \rightarrow P \rightarrow GL_n \rightarrow PGL_n$$
 	where the first morphism is an $A^1$-weak equivalence and the composition of the first two maps is the inclusion of $GL_{n-1} \hookrightarrow GL_n$. In particular, the restriction $H^{2,1}(BGL_n) \rightarrow H^{2,1}(BP) \cong H^{2,1}(BGL_{n-1})$ is an isomorphism. Therefore, we can consider the following diagram of long exact sequences
 	$$
 	\xymatrix{
 		H^{2,1}(BPGL_n) \cong 0 \ar@{->}[r] \ar@{=}[d] & 	H^{2,1}(BGL_n) \ar@{->}[r] \ar@{->}[d]^{\cong} &	H^{2,1}(\overline{N}) \ar@{->}[r] \ar@{->}[d] &	H^{3,1}(BPGL_n) \ar@{->}[r] \ar@{=}[d] &	H^{3,1}(BGL_n) \cong 0 \ar@{->}[d]\\
 		H^{2,1}(BPGL_n) \cong 0\ar@{->}[r]  & 	H^{2,1}(BP) \ar@{->}[r]  &	H^{2,1}(N) \ar@{->}[r]  &	H^{3,1}(BPGL_n) \ar@{->}[r]  &	H^{3,1}(BP) \cong 0 
 	}
 	$$
 	where $\overline{N}$ is $Cone(M(BGL_n \rightarrow BPGL_n)\rightarrow T)[-1]$ and $N$ is $Cone(M(BP \rightarrow BPGL_n)\rightarrow T)[-1]$ in $\DM_{eff}^{-}(BPGL_n)$. By five lemma the homomorphism $H^{2,1}(\overline{N}) \rightarrow H^{2,1}(N)$ is an isomorphism, which means that the class $x$ in $H^{3,1}(BPGL_n)$ is the image of a generator of $H^{2,1}(N) \cong \Z$.
 \end{rem}
 
 Let us denote $\ker(H^p_{\acute{e}t}(k,\mu_n^{\otimes p-1}) \rightarrow H^p_{\acute{e}t}(k(\seb(A)),\mu_n^{\otimes p-1}))$ simply by $\ker_p$. Also, in the following results, we denote by $H^{**}_{\acute et}(-)$ the \'etale motivic cohomology. So, in particular, we have that 
 $$H^{p,q}_{\acute et}(-,\Z/n) \cong H^p_{\acute{e}t}(-,\mu_n^{\otimes q}).$$
 
 \begin{prop}
 	We have the following isomorphisms: 
 	$$H^{p,q}(\X_A) \cong H^{p,q}(k)$$ 
 	for all $p \leq q$. Moreover,
 	$$H^{p,p-1}(\X_A) \cong 0$$
 	and 
 	$$H^{p+1,p-1}(\X_A) \cong \ker_p$$
 	for all $p$.
 \end{prop}
 \begin{proof}
 	By Bloch-Kato conjecture (see \cite[Theorems 6.16, 6.17 and 6.18]{voevodsky.motivic}), one has that $H^{p,q}(\X_A) \cong H^{p,q}_{\acute{e}t}(\X_A)$ and $H^{p,q}(k) \cong H^{p,q}_{\acute{e}t}(k)$ for $p \leq q+1$. Since $H^{p,q}_{\acute{e}t}(\X_A) \cong H^{p,q}_{\acute{e}t}(k)$, we get the first two isomorphisms of the statement.
 	
 	Regarding the last one, again by Bloch-Kato conjecture we have that
 	$$H^{p,p-1}(\X_A,\Z/n) \cong \ker_p$$
 	(see \cite[Remark 5.3]{voevodsky.motivic}). The short exact sequence
 	$$0 \rightarrow \Z \xrightarrow{\cdot n} \Z \rightarrow \Z/n \rightarrow 0$$
 	induces a long exact sequence in motivic cohomology
 	$$ \dots \rightarrow H^{p,p-1}(\X_A) \rightarrow H^{p,p-1}(\X_A,\Z/n) \rightarrow H^{p+1,p-1}(\X_A) \xrightarrow{\cdot n} H^{p+1,p-1}(\X_A) \rightarrow \dots.$$  
 	Therefore, since $H^{p,q}(\X_A)$ is $n$-torsion for $p \geq q+1$ and $H^{p,p-1}(\X_A) \cong 0$, one obtains
 	$$H^{p,p-1}(\X_A,\Z/n) \cong H^{p+1,p-1}(\X_A)$$
 	that conludes the proof.
 \end{proof}
 
 Recall from \cite[2.3.11 and Proposition 2.3.14]{smirnov.vishik} that in ${\mathcal H}_s(k)$
 $$\check C(\seb(A)) \cong (X_A \times EPGL_n)/PGL_n.$$ 
 Thus, there is a natural homomorphism $\alpha_A^*:H^{**}(BPGL_n) \rightarrow H^{**}(\X_A)$ induced by the map $\alpha_A: (X_A \times EPGL_n)/PGL_n \rightarrow BPGL_n$.
 
 \begin{prop}\label{xA}
 	We have that 
 	$$\alpha_A^*(x)=[A],$$
 	where $x$ is the canonical class in $H^{3,1}(BPGL_n)$ from Theorem \ref{comp} and $[A]$ is the Brauer class of $A$ in $H^{3,1}(\X_A) \cong \ker_2$.
 \end{prop}
 \begin{proof}
 	First note that, since $H^{3,1}(BPGL_n)$ and $H^{3,1}(\X_A)$ are both $n$-torsion, they are respectively isomorphic to $H^{2,1}(BPGL_n,\Z/n)$ and $H^{2,1}(\X_A,\Z/n)$. 
 	
 	The change of topology from Nisnevich to \'etale gives a commutative square
 	$$
 	\xymatrix{
 		H^{2,1}(BPGL_n,\Z/n) \ar@{->}[r] \ar@{->}[d] & H_{\acute{e}t}^{2,1}(BPGL_n,\Z/n) \cong H^2_{\acute{e}t}(BPGL_n,\mu_n) \ar@{->}[d]\\
 		H^{2,1}(\X_A,\Z/n) \cong \ker_2 \ar@{->}[r] & H_{\acute{e}t}^{2,1}(\X_A,\Z/n) \cong H^2_{\acute{e}t}(k,\mu_n)
 	}
 	$$
 	where the bottom horizontal morphism is the inclusion of $\ker_2$ in the Brauer group of $k$. By \cite[Theorem 1.2]{rolle}, the right vertical morphism maps the central extension
 	$$1 \rightarrow \mu_n \rightarrow SL_n \rightarrow PGL_n \rightarrow 1$$
 	(that is the image of $x$ under the top horizontal homomorphism: see Lemma \ref{ce} below for more details) to the class $[A]$ in the Brauer group. Hence, we deduce that the left vertical morphism does the same, as we aimed to show.
 \end{proof}
 
 \begin{prop}\label{pssb}
 	There exists a Postnikov system in $\DM_{eff}^-(k)$
 	$$
 	\xymatrix{
 		\X_A(n-1)[2n-2] \ar@{->}[r]  &M^{n-2} \ar@{->}[r] \ar@{->}[d] &  \dots \ar@{->}[r]   & M^2 \ar@{->}[r] \ar@{->}[d]	 &M^1 \ar@{->}[r] \ar@{->}[d]  &M(\seb(A)) \ar@{->}[d]\\
 		&	\X_A(n-2)[2n-4] \ar@{->}[ul]^{[1]} & & \X_A(2)[4] \ar@{->}[ul]^{[1]}  &	\X_A(1)[2] \ar@{->}[ul]^{[1]} & \X_A \ar@{->}[ul]^{[1]}
 	}
 	$$
 	where the composition $\X_A \rightarrow M^1[1] \rightarrow \X_A(1)[3]$ is the class $[A]$ in $H^{3,1}(\X_A)$.
 \end{prop}
 \begin{proof}
 	First, note that the projection to the second factor $(EPGL_n \times EP)/P \rightarrow EP/P$ has contractible fiber $EPGL_n$ over each simplicial component, so it is an isomorphism in ${\mathcal H}_s(k)$. On the other hand, the projection to the first factor $(EPGL_n \times EP)/P \rightarrow EPGL_n/P$ is also an isomorphism in ${\mathcal H}_s(k)$ since the $P$-torsor $PGL_n \rightarrow PGL_n/P\cong P^{n-1}$ is Zariski-locally trivial. So, we can fix $EPGL_n/P$ as a simplicial model for $BP$ and apply Proposition \ref{Serre} to the coherent morphism $EPGL_n/P \rightarrow EPGL_n/PGL_n$ with fiber $P^{n-1}$. This way we obtain a Postnikov system for the motive of $BP \rightarrow BPGL_n$ in $\DM_{eff}^-(BPGL_n)$
 	$$
 	\xymatrix{
 		T(n-1)[2n-2] \ar@{->}[r]  &N^{n-2} \ar@{->}[r] \ar@{->}[d] &  \dots \ar@{->}[r]   & N^2 \ar@{->}[r] \ar@{->}[d]	 &N^1 \ar@{->}[r] \ar@{->}[d]  &M(BP \rightarrow BPGL_n) \ar@{->}[d]\\
 		&	T(n-2)[2n-4] \ar@{->}[ul]^{[1]} & & T(2)[4] \ar@{->}[ul]^{[1]}  &	T(1)[2] \ar@{->}[ul]^{[1]} & T \ar@{->}[ul]^{[1]}
 	}
 	$$
 	where the composition $T \rightarrow N^1[1] \rightarrow T(1)[3]$ is $x$ by Remark \ref{long}.

 	Note that there is a cartesian square
 	$$
 	\xymatrix{
 		(X_A \times EPGL_n)/P \ar@{->}[r] \ar@{->}[d] & (X_A \times EPGL_n)/PGL_n \ar@{->}[d]\\
 		EPGL_n/P \ar@{->}[r] & EPGL_n/PGL_n.
 	}
 	$$
 	Since the $P$-torsor $X_A \rightarrow X_A/P$ is Zariski-locally trivial, we have a sequence of isomorphisms in ${\mathcal H}_s(k)$
 	$$(X_A \times EPGL_n)/P \leftarrow (X_A \times EP)/P \cong \check{C}(X_A \rightarrow X_A/P) \rightarrow X_A/P \cong \seb(A).$$
 	Therefore, by restricting the previous Postnikov system for $M(BP \rightarrow BPGL_n)$ along the functor $\DM_{eff}^-(BPGL_n) \rightarrow \DM_{eff}^-(\check C(\seb(A)))$ and by applying the forgetful functor to $\DM_{eff}^-(k)$, one obtains the needed Postnikov system for the motive of $\seb(A)$. Then, Proposition \ref{Serre 3} applied to the cartesian square above and Proposition \ref{xA} guarantee that the composition $\X_A \rightarrow M^1[1] \rightarrow \X_A(1)[3]$ is the class $[A]$ in $H^{3,1}(\X_A)$.
 \end{proof}
 
 \begin{thm}\label{sba}
 	There exists a strongly convergent spectral sequence
 	$$E_1^{p,q,s}= 
 	\begin{cases}
 	H^{p-2s,q-s}(\X_A) & 0 \leq s \leq n-1\\
 	0 & otherwise
 	\end{cases}
 	\Longrightarrow H^{p,q}(\seb(A))$$
 	with differentials $d_r^{p,q,s}:E_r^{p,q,s} \rightarrow E_r^{p+1,q,s-r}$. Moreover, the differential
 	$$d_1^{p,q,s}:H^{p-2s,q-s}(\X_A) \rightarrow H^{p-2s+3,q-s+1}(\X_A)$$
 	is the multiplication by $s[A]$ for $1 \leq s \leq n-1$.
 \end{thm}
 \begin{proof}
 	The spectral sequence is obtained by applying motivic cohomology to the Postnikov system in Proposition \ref{pssb}. The first differential is computed by using the same arguments of the proof of Theorem \ref{bpg}.
 \end{proof}
 
 \begin{cor}
 	For all $p \geq 3q+1$ we have that $H^{p.q}(\X_A) \cong 0$.
 \end{cor}
 \begin{proof}
 	The proof is the same as Corollary \ref{triv}.
 \end{proof}
 
 As an immediate consequence of the spectral sequence for the Severi-Brauer variety we obtain a description of the Chow group $CH^2(\seb(A))$.
 
 \begin{prop}
 	There is a short exact sequence
 	$$0 \rightarrow \coker(k^* \xrightarrow{\cdot [A]}\ker_3)   \rightarrow CH^2(\seb(A)) \rightarrow \Z \rightarrow 0.$$
 \end{prop}
 \begin{proof}
 	We use the spectral sequence of Theorem \ref{sba}. In this case the $E_1$-page is given by
 	$$E_1^{4,2,0} \cong H^{4,2}(\X_A) \cong \ker_3,$$
 	$$E_1^{4,2,1} \cong H^{2,1}(\X_A) \cong 0,$$
 	$$E_1^{4,2,2} \cong H^{0,0}(\X_A) \cong \Z.$$
 	In order to compute the $E_2$-page we also need
 	$$E_1^{3,2,1} \cong H^{1,1}(\X_A) \cong k^*,$$
 	$$E_1^{5,2,1} \cong H^{3,1}(\X_A) \cong \ker_2.$$
 	Note that $E_2^{4,2,2}$ is the kernel of the differential $d_1^{4,2,2}: E_1^{4,2,2} \rightarrow E_1^{5,2,1}$, i.e. $$E_2^{4,2,2} \cong \ker(\Z \xrightarrow{\cdot 2[A]} \ker_2) \cong \Z,$$
 	while $E_2^{4,2,0}$ is the cokernel of $d_1^{3,2,1}:E_1^{3,2,1} \rightarrow E_1^{4,2,0}$, i.e. the cokernel of the homomorphism $k^* \xrightarrow{\cdot [A]} \ker_3$. Since $E_2^{5,2,0} \cong E_1^{5,2,0} \cong H^{5,2}(\X_A)$ is $n$-torsion, we have that $E_{\infty}^{4,2,2} \cong E_3^{4,2,2} \cong \Z$. Moreover, $E_{\infty}^{4,2,0} \cong E_2^{4,2,0}$ and we get a filtration
 	$$F^{4,2,0} \hookrightarrow F^{4,2,1} \hookrightarrow F^{4,2,2} \cong H^{4,2}(\seb(A))$$
 	such that $E_{\infty}^{4,2,0} \cong F^{4,2,0}$, $E_{\infty}^{4,2,1} \cong F^{4,2,1}/F^{4,2,0}$ and $E_{\infty}^{4,2,2} \cong F^{4,2,2}/F^{4,2,1}$. Since $E_{\infty}^{4,2,1} \cong 0$, we obtain a short exact sequence
 	$$0 \rightarrow E_{\infty}^{4,2,0} \rightarrow F^{4,2,2} \rightarrow E_{\infty}^{4,2,2} \rightarrow 0$$
 	that is exactly the one we aimed to get.
 \end{proof}
 
 The previous result was already obtained by Peyre in \cite{peyre} by using different techniques. We have reported this new proof anyways as an example of a possible approach to the computation of Chow groups (and, more generally, motivic cohomology groups) of Severi-Brauer varieties by means of the spectral sequence in Theorem \ref{sba}. Of course, in order to get any information on the torsion of $CH^i(\seb(A))$ for $i \geq 3$ by using our spectral sequence one should first compute $H^{p,q}(\X_A)$ for $p \geq q+3$, which are generally unknown, at the best of our knowledge.
 
 \section{Torsion classes in $H^{**}(BPGL_n)$}
 
 In this section, following \cite{gu2} and \cite{gu3}, we find torsion classes in the motivic cohomology of $BPGL_n$ . This allows also to generalise some results about the Chow groups of $B_{\acute et}PGL_n$ from the complex numbers (see \cite[Theorem 1.1]{gu2} and \cite[Theorem 1]{gu3}) to more general fields. Indeed, we only require that the base field $k$ has characteristic not dividing $n$ and contains a primitive $n$th root of unity.
 
 First, let $n=p$ be an odd prime and consider the finite subgroup $C_p \times \mu_p$ of  $PGL_p$ described in \cite[Section 5]{vistoli}. Recall that $C_p$ is the subgroup of the symmetric group $S_p \subset PGL_p$ generated by the cycle $\sigma = (1 \: 2 \: \dots \: p)$ and $\mu_p$ is the subgroup of $PGL_p$ generated by the diagonal matrix $\rho =[\omega,\dots,\omega^{p-1},1]$, where $\omega$ is a primitive $p$th root of unity. Note that $\rho \sigma= \omega \sigma \rho$ in $GL_p$, so the two generators commute in $PGL_p$. The inclusion $\iota: C_p \times \mu_p \rightarrow PGL_p$ induces a homomorphism $B\iota^*:H^{**}(BPGL_p,\Z/p) \rightarrow H^{**}(B(C_p \times \mu_p),\Z/p)$. 
 
 Recall from Theorem \ref{bpg} that $H^{2,1}(BPGL_p)\cong 0$ and $H^{3,1}(BPGL_p) \cong \Z/p$, so the Bockstein homomorphism $\bock: H^{2,1}(BPGL_p,\Z/p) \rightarrow H^{3,1}(BPGL_p)$ is an isomorphism. Let $z$ be the class in $H^{2,1}(BPGL_p,\Z/p)$ such that $x=\bock(z)$. By \cite[Theorem 1.1]{rolle}, we know that $H^{2,1}_{\acute et}(BPGL_p,\Z/p) \cong H^2_{\acute et}(BPGL_p,\mu_p)$ is the group of central extensions of $PGL_p$ by $\mu_p$. 
 
 Before proceeding note also that the change of topology homomorphisms
 $$H^{2,1}(-,\Z/p) \rightarrow H^{2,1}_{\acute et}(-,\Z/p)$$
 and 
 $$H^{2,1}(-) \rightarrow H^{2,1}_{\acute et}(-)$$
 are respectively a monomorphism and an isomorphism for all simplicial schemes by \cite[Theorems 6.17 and 6.18]{voevodsky.motivic}.
 
 \begin{lem}\label{ce}
 	The change of topology homomorphism $H^{2,1}(BPGL_p,\Z/p) \rightarrow H^{2,1}_{\acute et}(BPGL_p,\Z/p)$ sends $z$ to the central extension
 	$$1 \rightarrow \mu_p \rightarrow SL_p \rightarrow PGL_p \rightarrow 1.$$
 \end{lem}
 \begin{proof}
 	We have a commutative square
 	$$
 	\xymatrix{
 		H^{2,1}(BPGL_p,\Z/p) \ar@{->}[r] \ar@{->}[d]_{\bock} & H_{\acute{e}t}^{2,1}(BPGL_p,\Z/p) \cong H^2_{\acute et}(BPGL_p,\mu_p) \ar@{->}[d]^{\bock}\\
 		H^{3,1}(BPGL_p)  \ar@{->}[r] & H^{3,1}_{\acute et}(BPGL_p)\cong H^2_{\acute et}(BPGL_p,\Gm) .
 	}
 	$$
 	Note that $H_{\acute{e}t}^{2,1}(BPGL_p) \cong H^{2,1}(BPGL_p) \cong 0$, so the Bockstein on the right is a monomorphism. Now, the statement immediately follows from the fact that $x=\bock(z)$ maps to the central extension
 	$$1 \rightarrow \Gm \rightarrow GL_p \rightarrow PGL_p \rightarrow 1$$
 	in $H^{3,1}_{\acute et}(BPGL_p) \cong H^2_{\acute et}(BPGL_p,\Gm)$.
 \end{proof}
 
 It follows from \cite[Lemma 2.3]{rolle} that $$H^{**}(B(C_p \times \mu_p),\Z/p) \cong H^{**}(k)[a,b,u,v]/(a^2=b^2=0)$$
 with $a$ and $b$ in bidegree $(0)[1]$, $u$ and $v$ in bidegree $(0)[2]$, such that $\beta(a)=u$ and $\beta(b)=v$ (where $\beta$ is the reduction mod $p$ of $\bock$).
 
 \begin{lem}
 	We have that
 	$$B\iota^*(z)=\lambda \tau ab,$$
 	where $\lambda $ is a non-zero element in $\Z/p$ and $\tau$ is the class in $H^{0,1}(k,\Z/p) \cong \mu_p(k)$ corresponding to the primitive $p$th root of unity $\omega$.
 \end{lem}	
 \begin{proof}
 	Note that $B\iota^*(z)$ is the class in $H^{2,1}(B(C_p \times \mu_p),\Z/p)$ that maps via the change of topology homomorphism to the central extension
 	$$1 \rightarrow \mu_p \rightarrow G \rightarrow C_p \times \mu_p \rightarrow 1$$
 	induced by the one in Lemma \ref{ce}. The class $B\iota^*(z)$ is non-zero since the previous extension is non-split, but restricts to a split extension both of $C_p$ and of $\mu_p$.
 	
 	By degree reasons $B\iota^*(z)$ has the following general form
 	$$B\iota^*(z)=\lambda \tau ab +\lambda_u \tau u +\lambda_v \tau v + \{r_a\} a + \{r_b\} b$$
 	where $\lambda$, $\lambda_u$ and $\lambda_v$ are in $\Z/p$ and $\{r_a\}$ and $\{r_b\}$ are in $K^M_{1}(k)/p$. Since $B\iota^*(z)$ restricts to zero both in $H^{2,1}(BC_p,\Z/p)$ and in $H^{2,1}(B\mu_p,\Z/p)$, we deduce that $\lambda_u=\lambda_v=0$ and $\{r_a\}=\{r_b\}=0$. Therefore, $B\iota^*(z)= \lambda \tau ab$ that concludes the proof.
 \end{proof}
 
 \begin{prop}
 	There are non-trivial classes $z_{p,k}$ in $H^{2p^{k+1}+1,p^{k+1}}(BPGL_p,\Z/p)$ for all $k \geq 0$.
 \end{prop}
 \begin{proof}
 	For all $k \geq 0$, define classes
 	$$z_{p,k}=\Pa^{p^k}\Pa^{p^{k-1}} \cdots \Pa^p\Pa^1\beta(z)$$
 	where $\Pa^i$ are the motivic Steenrod $p$th power operations constructed in \cite{voevodsky.reduced}.
 	
 	Since $\beta$ is a graded derivation \cite[(8.1)]{voevodsky.reduced}, we have that $\beta(z)$ is mapped by $B\iota^*$ to $\lambda \tau(ub-av)$. Now, we want to prove by induction on $k$ that
 	$$B\iota^*(z_{p,k})= \lambda \tau^{p^{k+1}} (u^{p^{k+1}}b-av^{p^{k+1}})$$
 	for any $k\geq 0$. For $k=0$, this reduces to
 	$$B\iota^*(z_{p,0})=B\iota^*(\Pa^1\beta(z))=\Pa^1(\lambda \tau(ub-av))=\lambda \tau^p (u^pb-av^p)$$
 	that follows from \cite[Proposition 9.7, Lemma 9.8 and Lemma 9.9]{voevodsky.reduced}. Suppose by induction hypothesis that $B\iota^*(z_{p,k-1})= \lambda \tau^{p^k} (u^{p^{k}}b-av^{p^{k}})$, then
 	$$B\iota^*(z_{p,k})=B\iota^*(\Pa^{p^k}(z_{p,k-1}))=\Pa^{p^k}(\lambda \tau^{p^k}(u^{p^{k}}b-av^{p^{k}}))=\lambda \tau^{p^{k+1}} (u^{p^{k+1}}b-av^{p^{k+1}})$$
 	again by \cite[Proposition 9.7, Lemma 9.8 and Lemma 9.9]{voevodsky.reduced}.
 	
 	Hence, $z_{p,k}$ is non-trivial for all $k$ that is what we aimed to show.
 \end{proof}
 
 \begin{prop}
 	There are non-trivial $p$-torsion classes $y_{p,k}$ in $H^{2p^{k+1}+2,p^{k+1}}(BPGL_p)$ for all $k \geq 0$.
 \end{prop}
 \begin{proof}
 	Define $y_{p,k}$ as $\bock(z_{p,k})$ where $\bock:H^{**}(BPGL_p,\Z/p) \rightarrow H^{**}(BPGL_p)$ is the Bockstein homomorphism. Note that the reduction mod $p$ of $y_{p,k}$ is nothing but $\beta(z_{p,k})$ which is non-trivial since maps to $\lambda \tau^{p^{k+1}} (u^{p^{k+1}}v-uv^{p^{k+1}})$ via $B\iota^*$. This finishes the proof.
 \end{proof}
 
 Note that the classes $z$, $\beta(z)$, $z_{p,k}$ and $\beta(z_{p,k})$ are not $\tau$-torsion, since their images under $B\iota^*$ are not $\tau$-torsion.
 
 Recall from \cite{morel.voevodsky} that the \'etale classifying space $B_{\acute et}G$ is defined as the object ${\mathrm R}\pi_*\pi^*(BG)$ in ${\mathcal H}_s(k)$, where $(\pi^*,{\mathrm R}\pi_*)$ is the couple of adjoint functors induced by the morphism of sites $\pi:(Sm/k)_{\acute et} \rightarrow (Sm/k)_{Nis}$. 
 
 \begin{prop}
 	There are non-trivial $p$-torsion classes ${\upsilon}_{p,k}$ in $CH^{p^{k+1}+1}(B_{\acute et}PGL_p)$ for all $k \geq 0$.
 \end{prop}
 \begin{proof}
 	By \cite[Theorem 6.17]{voevodsky.motivic} we have an isomorphism 
 	$$H^{2,2}(B_{\acute et}PGL_p,\Z/p) \rightarrow H^{2,2}(BPGL_p,\Z/p).$$
 	Let $\zeta$ be the class in $H^{2,2}(B_{\acute et}PGL_p,\Z/p)$ lifting $\tau z$ and define
 	$${\upsilon}_{p,k}=\bock \Pa^{p^k}\Pa^{p^{k-1}} \cdots \Pa^p\Pa^1\beta ({\zeta}).$$
 	The classes ${\upsilon}_{p,k}$ are non-trivial since their reductions mod $p$ map to $\tau \beta(z_{p,k})$.
 \end{proof}
 
 Let $p$ be an odd prime dividing $n$. Then, the diagonal map $\Delta: PGL_p \rightarrow PGL_n$ induces a homomorphism $H^{**}(BPGL_n) \rightarrow H^{**}(BPGL_p)$ that maps $x$ to $x$. Since the classes $z_{p,k}$, $y_{p,k}$ and $\upsilon_{p,k}$ for $BPGL_p$ are constructed starting from $\beta(z)$ (that is the reduction mod $p$ of $x$), we can define in the same way classes for $BPGL_n$. This immediately implies the following result.
 
 \begin{cor}\label{gentor}
 	For any odd prime $p$ dividing $n$ and $k \geq 0$, there are non-trivial $p$-torsion classes:
 	
 	1) $z_{p,k}$ in $H^{2p^{k+1}+1,p^{k+1}}(BPGL_n,\Z/p)$;
 	
 	2) $y_{p,k}$ in $H^{2p^{k+1}+2,p^{k+1}}(BPGL_n)$;
 	
 	3) ${\upsilon}_{p,k}$ in $CH^{p^{k+1}+1}(B_{\acute et}PGL_n)$.
 \end{cor}

\footnotesize{
	}

\noindent {\scshape Fachbereich Mathematik, Technische Universit{\"a}t Darmstadt}\\
fabio.tanania@gmail.com

\end{document}